\documentclass[11pt] {amsart}

\usepackage{amsthm}
\usepackage{amsxtra}
\usepackage{amssymb}
\usepackage{graphicx}
\usepackage{comment}

\usepackage{color}
\usepackage{amsfonts}
\usepackage{textcomp}
\usepackage{lmodern}

\DeclareMathOperator*{\supp}{supp}

\DeclareMathOperator*{\loc}{loc}
\DeclareMathOperator*{\comp}{comp}

\DeclareMathOperator*{\sgn}{sgn}

\newtheorem{theorem}{Theorem}

\numberwithin{prop}{section}

\numberwithin{corol}{section}
\numberwithin{equation}{section}

\newtheorem{lemma}{Lemma}
\numberwithin{lemma}{section}

\numberwithin{conjecture}{section}

{\theoremstyle{definition}

\numberwithin{defin}{section}
}
\numberwithin{figure}{section}

\renewcommand{\Re}{\mathop{\rm Re}\nolimits}
\renewcommand{\Im}{\mathop{\rm Im}\nolimits}

\newcommand{\bl}{\begin{flushleft}}
\newcommand{\el}{\end{flushleft}}
\newcommand{\br}{\begin{flushright}}
\newcommand{\ert}{\end{flushright}}
\newcommand{\bc}{\begin{center}}
\newcommand{\ec}{\end{center}}

\newcommand{\mcal}[1]{\mathcal{#1}}

\newcommand{\complex}{\mathbb{C}}

\newcommand{\numList}{\begin{enumerate}}
\newcommand{\enumList}{\end{enumerate}}

\newcommand{\e}{\epsilon}

\newcommand{\re}{\mathbb{R}}

\newcommand{\la}{\langle}
\newcommand{\ra}{\rangle}

\newcommand{\Dloc}{{\mcal D}_{\loc}}
\newcommand{\hf}{\frac 12}
\newcommand{\Kbar}{\tilde{K}}
\newcommand{\sph}{\mathbb{S}}
\newcommand{\LOmega}{d_\Omega}
\newcommand{\LGamma}{d_\Gamma}
\newcommand{\Lchi}{d_\chi}

\newenvironment{remark}[1][Remarks:]{\begin{trivlist}
\item[\hskip \labelsep {\bfseries #1}]\end{trivlist}\begin{itemize}}{\end{itemize}}

\numberwithin{equation}{section}

\title [Restriction Bounds for the Free Resolvent]{Restriction bounds for the free resolvent and resonances in lossy scattering}
\author[J. Galkowski]{Jeffrey Galkowski}
\address{Mathematics Department, University of California, Berkeley, 
CA 94720, USA}
\email{jeffrey.galkowski@math.berkeley.edu}
\author[H. Smith]{Hart Smith}
\address{Mathematics Department, University of Washington, Seattle, WA 98195, USA}
\email{hart@math.washington.edu}

\begin{document}

\begin{abstract}

We establish high energy $L^2$ estimates for the restriction of the free Green's function to hypersurfaces in $\re^d$. As an application, we estimate the size of a logarithmic resonance free region for scattering by potentials of the form $V\otimes \delta_{\Gamma}$, where $\Gamma\subset\re^d$ is a finite union of compact subsets of embedded hypersurfaces. In odd dimensions we prove a resonance expansion for solutions to the wave equation with such a potential.

\end{abstract}

\maketitle
\section{Introduction} Scattering by potentials is used in math and physics to study waves in many physical systems  (see for example \cite{Burke}, \cite{Lax}, \cite{ZwAMS}, \cite{ZwScat} and the references therein). Examples include acoustics in concert halls and scattering of light by black holes. 
 One case of recent interest is scattering in quantum corrals that are constructed using scanning tunneling microscopes \cite{Heller} \cite{Crommie}. One model for this system is that of a delta function potential on the boundary of a domain $\Omega\subset \re^d$ (see for example \cite{Aligia}, \cite{Heller}, \cite{Crommie}). In this paper, we study scattering by such a delta function potential on hypersurfaces $\Gamma\subset \re^d$.

We assume that $\Gamma\subset \re^d$ is a finite union of compact subsets of embedded $C^{1,1}$ hypersurfaces; that is, it is a union of compact subsets of graphs of $C^{1,1}$ functions. The Bunimovich stadium is an example of a domain in two dimensions which has boundary that is $C^{1,1}$, but not $C^2$.
We let $\delta_{\Gamma}$ denote the surface measure on $\Gamma$, considered as a distribution on $\re^d$, and take
$V$ to be a bounded, self-adjoint operator on $L^2(\Gamma)$. For
$u\in H^1_{\loc}(\re^d)$, we then define 
$(V\otimes\delta_{\Gamma})u:=(Vu|_{\Gamma})\delta_{\Gamma}$.

Resonances are defined as poles of the meromorphic continuation from $\Im \lambda \gg 0$ of the resolvent 
$$
R_V(\lambda)=(-\Delta_{V,\Gamma}-\lambda^2)^{-1}\,,
$$
 where $-\Delta_{V,\Gamma}$ is the unbounded self-adjoint operator
 $$
-\Delta_{V,\Gamma}:=-\Delta +V\otimes \delta_{\Gamma}
$$
(See Section \ref{sec:formalDefinition} for the formal definition of $-\Delta_{V,\Gamma}$). If the dimension $d$ is odd, $R_V(\lambda)$ admits a meromorphic continuation to the entire complex plane, and to the logarithmic covering space of $\complex\setminus \{0\}$ if $d$ is even (see Section \ref{sec:resExpand}). 

The imaginary part of a resonance gives the  decay rate of the associated resonant states. Thus, resonances close to the real axis give information about long term behavior of waves. In particular, since the seminal work of Lax-Phillips \cite{Lax} and Vainberg \cite{Vain}, resonance free regions near the real axis have been used to understand decay of waves.

In this paper, we demonstrate the existence of a resonance free region for delta function potentials on a very general class of $\Gamma$.

\begin{theorem}
\label{thm:resFreeEasy}
Let $\Gamma\subset \re^d$ be a finite union of compact subsets of embedded $C^{1,1}$ hypersurfaces, and suppose $V$ is a self-adjoint operator on $L^2(\Gamma)$.

Then for all $\e>0$ there exists $R>0$ such  that, if $\lambda$ is a resonance for $-\Delta_{V,\Gamma}$, then
\begin{equation}
\label{eqn:resFreeRegion}
\Im \lambda \leq -\Bigl(\tfrac 12\LGamma^{-1}-\e\Bigr)\log (|\!\Re\lambda|)\quad\mathrm{if}\quad |\!\Re \lambda| \geq R\,,
\end{equation}
where $\LGamma$ is the diameter of the convex hull of $\Gamma$.
If $\Gamma$ can be written as a finite union of strictly convex $C^{2,1}$ hypersurfaces, then we can replace $\frac12$ by $\frac{2}{3}$ in \eqref{eqn:resFreeRegion}.
\end{theorem}

\begin{remark}
\item These bounds on the size of the resonance free region are not generally optimal, for example in the case that $\Gamma=\partial B(0,1)\subset \re^2$. In \cite{Galk}, the first author uses a microlocal analysis of the transmission problem \eqref{eqn:main} to obtain sharp bounds in the case that $\Gamma=\partial\Omega$ is $C^\infty$ with $\Omega$ strictly convex.  
\item In the smooth, strictly convex case, scattering in other types of transmission problems was considered in \cite{Card} and \cite{PopVod}.
\end{remark}

\medskip

Let $R_0(\lambda)$ be the analytic continuation of the outgoing free resolvent $(-\Delta-\lambda^2)^{-1}$, defined initially for $\Im\lambda>0$. 
Theorem \ref{thm:resFreeEasy} follows from bounds on an operator related to the free resolvent. In particular, we study the restriction of $R_0(\lambda)$ to hypersurfaces $\Gamma \subset\re^d$. Let $\gamma$ denote restriction to $\Gamma$, and $\gamma^*$ the inclusion map $f\mapsto f\delta_\Gamma$.
Let $G(\lambda):L^2(\Gamma)\to L^2(\Gamma)$ be obtained by restricting the kernel $G_0(\lambda,x,y)$ of $R_0(\lambda)$ to $\Gamma$,
\begin{equation*}
G(\lambda):=\gamma \, R_0(\lambda) \, \gamma^*\,.
\end{equation*}
Theorem \ref{thm:resFreeEasy} will follow as a consequence of the following theorem,

\begin{theorem}
\label{thm:optimal}
Let $\Gamma\subset\re^d$ be a finite union of compact subsets of embedded $C^{1,1}$ hypersurfaces. Then $G(\lambda)$ is a compact operator on $L^2(\Gamma)$, and
\begin{equation}
\label{eqn:optimalFlat}
\|G(\lambda)\|_{L^2(\Gamma)\to L^2(\Gamma)}\leq C\,\la\lambda\ra^{-\frac 12}\,\log\la\lambda\ra\,e^{\LGamma(\Im \lambda)_-}\,,
\end{equation}
where $\LGamma$ is the diameter of the convex hull of $\Gamma$.
Moreover, if $\Gamma$ is a finite union of compact subsets of strictly convex $C^{2,1}$ hypersurfaces, then 
\begin{equation}
\label{eqn:optimalConvex}
\|G(\lambda)\|_{L^2(\Gamma)\to L^2(\Gamma)}\leq C\,\la\lambda\ra^{-\frac 23}\,\log\la\lambda\ra\,e^{\LGamma(\Im \lambda)_-}\,.
\end{equation}

\end{theorem}
Here we set $\la\lambda\ra=(2+|\lambda|^2)^{\frac 12}$, and $(\Im\lambda)_-=\max(0,-\Im\lambda)\,.$
Compactness follows easily by Rellich's embedding theorem, or the bounds on $G_0(\lambda,x,y)$ in Section \ref{sec:boundGreenFunc}.
The powers $\frac 12$ and $\frac 23$ in \eqref{eqn:optimalFlat} and \eqref{eqn:optimalConvex}, respectively, are in general optimal. This follows from the fact that the corresponding estimates for the restriction of eigenfunctions in Section \ref{sec:upperhalfplane} are the best possible. However, it is likely that the factor of $\log\la\lambda\ra$ is not needed. In Section \ref{sec:twodimensions} we prove estimate \eqref{eqn:optimalFlat} in dimension two without it. Also, for the flat case in general dimensions, the estimate \eqref{eqn:optimalFlat} holds without it. We also expect that estimate \eqref{eqn:optimalConvex} holds for $C^{1,1}$ strictly convex hypersurfaces, but do not pursue that here.

In the case that $\Im\lambda\ge|\lambda|^{\frac 12}$, respectively $\Im\lambda\ge|\lambda|^{\frac 23}$,
the above bounds can be improved upon.
\begin{theorem}
\label{thm:optimalLargeIm}
Let $\Gamma\subset\re^d$ be a finite union of compact subsets of embedded $C^{1,1}$ hypersurfaces. Then for $\Im\lambda>0$,
\begin{equation*}
\|G(\lambda)\|_{L^2(\Gamma)\to L^2(\Gamma)}\leq C\,\la\Im\lambda\ra^{-1}\,.
\end{equation*}
\end{theorem}

We next use the results above to analyze the long term behavior of waves scattered by the potential $V\otimes \delta_{\Gamma}$. Theorem \ref{thm:resFreeEasy} implies in particular that there are only a finite number of resonances in the set
$\Im\lambda> -A\,,$ for any $A<\infty$.
We give a resonance expansion for the wave equation
\begin{equation}
\label{eqn:waveProblem}
\bigl(\partial_t^2-\Delta +V\otimes \delta_{ \Gamma}\bigl)u=0\,,\quad u(0,x)=0,\quad \partial_tu(0,x)=g\in L^2_{\comp}\,,
\end{equation}
with wave propagator $U(t)$ defined using the functional calculus for $-\Delta_{V,\Gamma}$. Let $m_R(\lambda)$ be the multiplicity of the pole of $R_V(\lambda)$ at $\lambda$, that is the dimension of the set of resonant states with resonance $\lambda$, and let $\mcal{D}_N$ be the domain of $(-\Delta_{V,\Gamma})^N$. 

\begin{theorem}
\label{thm:resExpand}
Let $d>0$ be odd, and assume that $\Gamma\subset \re^d$ is a finite union of compact subsets of embedded $C^{1,1}$ hypersurfaces, and that $V$ is a self-adjoint operator on $L^2(\Gamma)$.

Let $0> -\mu_1^2>\cdots> -\mu_N^2$ and $0 < \nu_1^2< \cdots < \nu_M^2$ be the nonzero eigenvalues of $-\Delta_{V,\Gamma}$, and $\{\lambda_j\}$ the resonances with $\Im \lambda < 0$.
Then for any 
$A>0$ and $g\in L^2_{\comp}\,,$ the solution $U(t)g$ to \eqref{eqn:waveProblem} admits an expansion
\begin{multline}
\label{eqn:resExpand}
U(t) g=\sum_{k=1}^N \; (2\mu_k)^{-1}e^{t\mu_k}\Pi_{\mu_k}g
+t \,\Pi_0 g+\mcal P_0g+\sum_{k=1}^M (2\nu_k)^{-1}\sin(t\nu_k)\Pi_{\nu_k}g
\\
+\sum_{\Im \lambda_j>-A}\sum_{k=0}^{m_R(\lambda_j)-1}e^{-it\lambda_j}\,t^k\, \mcal P_{\lambda_j,k}g +E_A(t)g\,,
\end{multline}
where $\Pi_{\mu_k}$ and $\Pi_{\nu_k}$ respectively denote the projections onto the $-\mu_k^2$ and $\nu_k^2$ eigenspaces. The maps $\mcal P_{\lambda_j,k}$ are bounded
from $L^2_{\comp}\rightarrow \Dloc$, and
$\mcal P_0$ is a symmetric map to the $0$-resonances.

The operator $E_A(t)\,:\,L^2_{\comp}\rightarrow L^2_{\loc}$ has the following property: for any
$\chi\in C_c^\infty(\re^d)$ equal to $1$ on a neighborhood of $\Gamma$, and any $N\ge 0$, there exists $T_{A,\chi,N}<\infty$ so that
$$
\|\chi E_A(t)\chi\|_{L^2\to \mcal D_N}\leq C_{A,\chi,N}\,e^{-At}\,, \qquad t>T_{A,\chi,N}\,.
$$
\end{theorem}

Under the assumption that $\Gamma=\partial\Omega$ for a bounded open domain $\Omega\subset\re^d$, and that $V$ and $\partial\Omega$ satisfy higher regularity assumptions, we obtain estimates for $\chi E_A(t)\chi g$ in the spaces 
$$
\mcal E_N:= H^1(\re^d)\cap (H^N(\Omega)\oplus H^N(\re^d\setminus\overline{\Omega}))\,,\qquad N\ge 1\,.
$$
If $\partial\Omega$ is of $C^{1,1}$ regularity, and $V$ is bounded on $H^{\frac 12}(\partial\Omega)$, then we show $\mcal D_1=\mcal D\subset\mcal E_2$, and convergence in $\mcal E_2$ follows from Theorem \ref{thm:resExpand}. For smooth boundaries we show the following.

\begin{theorem}\label{thm:higherorder}
Suppose that $\Gamma=\partial\Omega$ is $C^\infty$ and that $V$ is bounded on $H^s(\partial\Omega)$ for all $s$. Then the operator $E_A(t)$ defined in \eqref{eqn:resExpand} has the following property: for any
$\chi\in C_c^\infty(\re^d)$ equal to $1$ on a neighborhood of $\overline\Omega$, and integer $N\ge 1$, there exists $T_{A,\chi,N}<\infty$ so that
$$
\|\chi E_A(t)\chi\|_{L^2\to \mcal E_N}\leq C_{A,\chi,N}\,e^{-At}\,, \qquad t>T_{A,\chi,N}\,.
$$
\end{theorem}

In addition to describing resonances as poles of the meromorphic continuation of the resolvent, we will give a more concrete description of resonances in Sections \ref{sec:equivalence} and \ref{sec:resExpand}. We show that
$\lambda$ is a resonance of the system if and only if there is a nontrivial $\lambda$-outgoing solution $u\in\Dloc$ to the equation
\begin{equation}
\label{eqn:mainDelta}(-\Delta -\lambda^2+V\otimes\delta_{\Gamma})u=0\,,
\end{equation}
where we define
$$
\Dloc=\{u\,:\,\chi u\in \mcal D\:\,\text{whenever}\; \chi \in C^\infty_c(\re^d)\;\text{and}\;\chi=1\;\text{on a neighborhood of}\;\Gamma\bigr\}\,.
$$
Here we say that $u$ is $\lambda$-\emph{outgoing} if for some $R<\infty$, and 
some compactly supported distribution $g$, we can write
$$
u(x)=\bigl(R_0(\lambda)g\bigr)(x)\quad\text{for}\quad|x|\ge R\,.
$$

Moreover, if we assume that $\Gamma=\partial\Omega$ for a $C^{1,1}$ domain $\Omega$, and that $V:H^{\frac 12}(\partial\Omega)\to H^{\frac 12}(\partial\Omega)$, we show this is equivalent to solving the following transmission problem with $u\in H^1_{\loc}(\re^d)$, and with $u|_{\Omega}=u_1\in H^2(\Omega)$, $u|_{\re^d\setminus\overline{\Omega}}=u_2\in H^2_{\loc}(\re^d\setminus\overline\Omega)$,
\begin{equation}
\label{eqn:main}
\begin{cases}(-\Delta -\lambda^2)u_1=0&\text{ in }\Omega\\
(-\Delta -\lambda^2)u_2=0 & \text{ in }\re^d\setminus \overline{\Omega}\\
u_1=u_2&\text{ on }\partial \Omega\\
\partial_\nu u_1+\partial_{\nu '}u_2+Vu_1=0&\text{ on }\partial \Omega\\
u_2\;\lambda\text{-outgoing}
\end{cases}
\end{equation}
Here, $\partial_\nu$ and $\partial_{\nu'}$ are respectively the interior and exterior normal derivatives of $u$ at $\partial\Omega$.

The outline of this paper is as follows. In Section \ref{sec:preliminaries} we present the definition of $-\Delta_{V,\Omega}$ and its domain, as well as some preliminary bounds on the outgoing Green's function $G_0(\lambda,x,y)$. In Section \ref{sec:twodimensions} we give a simple proof of Theorem \ref{thm:optimal} for $d=2$. In Section \ref{sec:upperhalfplane} we establish Theorem \ref{thm:optimal} for $\Im \lambda \geq 0$ in all dimensions, deriving the estimates from restriction estimates for eigenfunctions of the Laplacian. We include a proof of the desired restriction estimate for hypersurfaces of regularity $C^{1,1}$, since the result appears new, and also provide the proof of Theorem \ref{thm:optimalLargeIm}. In Section \ref{sec:lowerhalfplane} we complete the proof of Theorem \ref{thm:optimal} for $\Im\lambda<0$ using the Phragm\'en-Lindel\"of theorem. In Section \ref{sec:equivalence} we demonstrate the meromorphic continuation of $R_V(\lambda)$, give the proof of Theorem \ref{thm:resFreeEasy}, and relate resonances to solvability of an equation on $\Gamma$, and for $\Gamma=\partial\Omega$ to solvability of \eqref{eqn:main}. In Section \ref{sec:resExpand} we give more detailed structure of the meromorphic continuation of $R_V(\lambda)$. We establish mapping bounds for compact cutoffs of $R_V(\lambda)$, and use these to prove Theorems \ref{thm:resExpand} and \ref{thm:higherorder} by a contour integration argument. In Section \ref{sec:transmission} we prove a needed transmission property estimate for boundaries of regularity $C^{1,1}$.

\medskip

\noindent
{\sc Acknowledgements.} The authors would like to thank Maciej Zworski for valuable guidance and discussions and Semyon Dyatlov for many useful conversations. 
This material is based upon work supported by the National Science Foundation under Grants 
DGE-1106400, DMS-1201417, and DMS-1161283. This work was partially supported by a grant from the Simons Foundation (\#266371 to Hart Smith).

\newpage

\section{Preliminaries}
\label{sec:preliminaries}
\subsection{Determination of $-\Delta_{V,\Gamma}$ and its domain.}
\label{sec:formalDefinition}
We define the operator $-\Delta_{V,\Gamma}$ using the symmetric, densely defined quadratic form 
$$
Q_{V,\Gamma}(u,w):=\la \nabla u,\nabla w\ra_{L^2(\re^d)} + \la Vu,w\ra_{L^2(\Gamma)}\,
$$
with domain $H^1(\re^d)\subset L^2(\re^d)$.

For $\Gamma$ a finite union of compact subsets of $C^{1,1}$ hypersurfaces (indeed Lipschitz hypersurfaces suffice), we can bound
$$
\|u\|_{L^2(\Gamma)}\le C\,\|u\|^{\frac 12}_{L^2}\|u\|^{\frac 12}_{H^1}\le C\,\e\, \|u\|_{H^1}+C\,\e^{-1}\|u\|_{L^2}\,.
$$
It follows that there exist $c\,, C>0$ such that 
$$
|Q_{V,\Gamma}(u,w)|\leq \|u\|_{H^1}\|w\|_{H^1}\quad\text{ and }\quad c\,\|u\|_{H^1}^2\leq Q_{V,\Gamma}(u,u)+C\|u\|_{L^2}^2\,.
$$
By Reed-Simon \cite[Theorem VIII.15]{RS}, $Q_{V,\Gamma}(u,w)$ is determined by a unique self-adjoint operator
$-\Delta_{V,\Gamma}$, with domain $\mcal{D}$ consisting of $u\in H^1$ for which $Q_{V,\Gamma}(u,w)\le C\|w\|_{L^2}$.

For $u\in \mcal D$, by the Riesz representation theorem, we have $Q_{V,\Gamma}(u,w)=\la f,w\ra$ for some $f\in L^2(\re^d)$, and taking $w\in C_c^\infty(\re^d)$ shows that in the sense of distributions 
\begin{equation}
\label{eqn:LaplaceDistributional}
-\Delta u+(V u|_{\Gamma})\delta_{\Gamma}=f\,.
\end{equation}
Conversely, if $u\in H^1(\re^d)$ and \eqref{eqn:LaplaceDistributional} holds for some $f\in L^2(\re^d)$, then by density of $C_c^\infty\subset H^1$ we have $Q_{V,\Gamma}(u,w)=\la f,w\ra$ for $w\in H^1(\re^d)$, hence $u\in\mcal D$, and $-\Delta_{V,\Gamma} u$ is given by the left hand side of \eqref{eqn:LaplaceDistributional}. We thus can write (up to a constant of proportionality)
$$
\|u\|_{\mcal D} = \|u\|_{H^1}+ \|\Delta_{V,\Gamma}u\|_{L^2}\,,
$$
where finiteness of the second term carries the assumption that $\Delta_{V,\Gamma}u\in L^2$.

The domain $\mcal D_N\subset \mcal D$ is defined for $N\ge 1$ by the condition $\Delta_{V,\Gamma}u\in \mcal D_{N-1}$, and we will recursively define (consistent up to constants with the definition using the functional calculus)
$$
\|u\|_{\mcal D_N} = \|u\|_{H^1}+ \|\Delta_{V,\Gamma}u\|_{\mcal D_{N-1}}\,,\qquad N\ge 1\,.
$$

Suppose that $\chi\in C_c^\infty(\re^d\setminus\Gamma)$ and that $u\in H^1$ solves \eqref{eqn:LaplaceDistributional}. Then, 
$$
\Delta (\chi u)= \chi f +2\nabla \chi \cdot \nabla u+(\Delta \chi)u\in L^2(\re^d)\,.
$$
Hence,
$$
\|\chi u\|_{H^2}\le C_\chi \|u\|_{\mcal D}\,.
$$
That is, $\mcal D\subset H^1(\re^d)\cap H^2_{\loc}(\re^d\setminus\Gamma)$, with continuous inclusion.
Similar arguments show that 
$$
\mcal{D}_N\subset H^1(\re^d)\cap H^{2N}_{\loc}(\re^d\setminus\Gamma)\,.
$$

The behavior of $u$ near $\Gamma$ may be more singular. For general $V$ acting on $L^2(\Gamma)$, from
\eqref{eqn:LaplaceDistributional} and the fact that $(Vu|_\Gamma)\delta_\Gamma\in H^{-\frac 12-\e}(\re^d)$ for all $\e>0$, we conclude that $u\in H^{\frac 32-\e}(\re^d).$ However, under additional assumptions on $V$ and $\Gamma$ we can give a full description of $\mcal{D}$ near $\Gamma$. 

For the purposes of the remainder of this section we assume that $\Gamma=\partial\Omega$ for some bounded open domain $\Omega\subset \re^d$, and that $\partial\Omega$ is a $C^{1,1}$ hypersurface; that is, locally $\partial\Omega$ can be written as the graph of a $C^{1,1}$ function. We assume also that $V:H^{\frac 12}(\partial\Omega)\to H^{\frac 12}(\partial\Omega)$. Then since $u\in H^1(\re^d)$, $Vu|_{\partial\Omega}\in H^{\frac 12}(\partial\Omega)$.
Hence, Lemma \ref{lem:transmission} combined with \eqref{eqn:LaplaceDistributional} shows that
$$\mcal D\subset \mcal E_2=H^1(\re^d)\cap (H^2(\Omega)\oplus H^2(\re^d\setminus\overline{\Omega}))\,,
$$
with continuous inclusion.
We remark that
$H^2(\Omega)$ and $H^2(\re^d\setminus\overline{\Omega})$ can be identified as restrictions of $H^2(\re^d)$ functions; see \cite{Calderon} and \cite[Theorem VI.5]{Stein}. Thus, if $u\in \mcal D$ both $u$ and
its first derivatives have well defined traces on $\partial\Omega$, respectively of regularity $H^{\frac 32}(\partial\Omega)$ and $H^{\frac 12}(\partial\Omega)$.

For $w \in H^1(\re^d)$ and $u\in H^1(\re^d)\cap (H^2(\Omega)\oplus H^2(\re^d\setminus\overline{\Omega}))$,
it follows from Green's identities that
$$
Q_{V,\partial\Omega}(u,w)=\la -\Delta u,w\ra_{\Omega}+\la -\Delta u,w\ra_{\re^d\setminus\overline\Omega}+
\la \partial_{\nu}u+\partial_{\nu'}u+Vu,w\ra_{\partial\Omega}\,,
$$
where $\partial_\nu$ and $\partial_{\nu'}$ denote the exterior normal derivatives from $\Omega$ and $\re^d\setminus\overline{\Omega}$.
Thus, in the case that $V:H^{\frac 12}(\partial\Omega)\to H^{\frac 12}(\partial\Omega)$, we can completely characterize the domain $\mcal D$ of the self-adjoint operator $-\Delta_{V,\partial\Omega}$ as
\begin{equation}\label{eqn:domainChar}
\mcal D=\bigl\{u\in H^1(\re^d)\cap \bigl(H^2(\Omega)\oplus H^2(\re^d\setminus\overline{\Omega})\bigr)\quad
\text{such that}\quad \partial_{\nu}u+\partial_{\nu'}u+Vu=0\,\bigr\}\,,
\end{equation}
in which case $\Delta_{V,\partial\Omega}u=\Delta u|_\Omega\oplus \Delta u|_{\re^d\setminus\overline\Omega}$.

\subsection{Bounds on Green's function}
\label{sec:boundGreenFunc}
We conclude this section by reviewing bounds on the convolution kernel $G_0(\lambda,x,y)$ associated to the operator $R_0(\lambda)$. It can be written in terms of the Hankel functions
of the first kind,
$$
G_0(\lambda,x,y)=C_d\,\lambda^{d-2}\,\bigl(\lambda|x-y|\bigr)^{-\frac{d-2}2}\,H^{(1)}_{\frac d2-1}\bigl(\lambda |x-y|\bigr)\,,
$$
for some constant $C_d$. If $d$ is odd, this can be written as a finite expansion
$$
G_0(\lambda,x,y)=\lambda^{d-2}\,e^{i\lambda|x-y|}
\sum_{j=\frac{d-1}2}^{d-2}\frac {c_{d,j}}{\bigl(\lambda|x-y|\bigr)^j}\,.
$$
For $x\ne y$ this form extends to $\lambda\in\complex$, and defines the analytic extension of $R_0(\lambda)$.
In particular, we have the upper bounds
\begin{equation}\label{R0bounds}
|G_0(\lambda,x,y)|\lesssim
\begin{cases}
|x-y|^{2-d}\,,& |x-y|\le|\lambda|^{-1}\,,\\
e^{-\Im\lambda|x-y|}\,|\lambda|^{\frac{d-3}2}\,|x-y|^{\frac{1-d}2}\,,&|x-y|\ge|\lambda|^{-1}\,.
\rule{0pt}{14pt}
\end{cases}
\end{equation}
For $d$ even, and $d\ne 2$, the bounds \eqref{R0bounds} hold for $\Im\lambda>0$, as well as for the analytic extension to $-\pi\le\arg\lambda\le 2\pi$. For $-\pi<\arg\lambda< 2\pi$ this
follows by the asymptotics of $H^{(1)}_n(z)$; see for example \cite[(9.2.3)]{AS}.
To see that it extends to the closed region, we use the relation (valid in all dimensions)
$$
G_0(e^{i\pi}\lambda,x,y)-G_0(\lambda,x,y)=\frac{i\,\lambda^{d-2}}{(2\pi)^{d-1}}\int_{\sph^{d-1}} e^{i\lambda\la x-y,\omega\ra}\,d\omega=C_d\,\lambda^{d-2}\,\bigl(\lambda|x-y|\bigr)^{-\frac{d-2}2}\,J_{\frac d2-1}\bigl(\lambda|x-y|\bigr)
$$
where $d\omega$ is surface measure on the unit sphere $\sph^{d-1}\subset\re^d$, and $e^{i\pi}$ indicates analytic continuation through positive angle $\pi$. The bounds \eqref{R0bounds} then follow from the asymptotics of $J_n(z)$ and the bounds for $\Im\lambda\ge 0$. We also note as a consequence of the above that, for $\lambda\in\re\setminus\{0\}$, and any sheet of the continuation in even dimensions,
\begin{equation}\label{R0diff}
G_0(e^{i\pi}\lambda,x,y)-G_0(\lambda,x,y)=2\pi i\,(\sgn\lambda)^d\,|\lambda|^{-1}\widehat{\delta_{\sph^{d-1}_\lambda}}(x-y)\,,
\end{equation}
where $\delta_{\sph^{d-1}_\lambda}$ denotes surface measure on the sphere $|\xi|=|\lambda|$ in
$\re^d$.

In the case that $d=2$,  in \eqref{R0bounds} one need replace $|x-y|^{2-d}$ by $-\ln |x-y|$ in the bounds for $|x-y|\le|\lambda|^{-1}$. However, for our purposes we use only the following global bound in case $d=2$,
\begin{equation*}
|G_0(\lambda,x,y)|\lesssim e^{-\Im\lambda|x-y|}\,|\lambda|^{-\frac 12}\,|x-y|^{-\frac 12}\,,
\qquad d=2\,,\quad -\pi\le\arg\lambda\le 2\pi\,.
\end{equation*}

Finally, we observe that since $G_0(\lambda,x,y)$ is smooth away from the diagonal, with
the singularity at $x=y$ integrable over $\Gamma$, it follows that $G(\lambda)$ is a compact operator on $L^2(\Gamma)$ for every $\lambda$.
This also follows from \eqref{eqn:GHhalfbound}.


\section{Estimates for $d=2$}\label{sec:twodimensions}

In this section we give an elementary proof of estimate \eqref{eqn:optimalFlat} of Theorem \ref{thm:optimal} for $d=2$. Indeed, we can prove the following stronger result,
\begin{theorem}
Under the conditions of Theorem \ref{thm:optimal} the following holds, for $-\pi\le\arg\lambda\le 2\pi\,,$
$$
\|G(\lambda)f\|_{L^2(\Gamma)}\le 
\begin{cases}
C\,\la\lambda\ra^{-\frac 12}\la\Im\lambda\ra^{-\frac 12}\,\|f\|_{L^2(\Gamma)}\,,& \Im\lambda\ge 0\,,\\
C\,\la\lambda\ra^{-\frac 12}\,e^{-\LGamma\Im \lambda}\,\|f\|_{L^2(\Gamma)}\,,& \Im\lambda\le 0\,.
\rule{0pt}{16pt}
\end{cases}
$$
\end{theorem}
\begin{proof}
We use the kernel bounds, for $x,y$ in a bounded set,
$$
|G_0(\lambda,x,y)|\le C\,e^{-\Im\lambda|x-y|}\,\la\lambda\ra^{-\frac 12}\,|x-y|^{-\frac 12}\,.
$$
By the Schur test and symmetry of the kernel, the operator norm is bounded by the following
$$
\sup_x\int_\Gamma |G_0(\lambda,x,y)|\,d\sigma(y)\,.
$$
First consider $\Im\lambda\le 0$. Then $e^{-\Im\lambda|x-y|}\le e^{-\LGamma\Im\lambda}$ for $x,y\in\Gamma$, and since $\Gamma$ is a finite union of subsets of $C^{1,1}$ hypersurfaces the desired bound follows from the following, which holds if $\Gamma$ is a bounded subset of the graph of a Lipschitz function,
$$
\sup_x\int_\Gamma |x-y|^{-\frac 12}\,d\sigma(y)\le C\,.
$$
For $\Im\lambda\ge 0$, we use instead the bound
$$
\sup_x\int_\Gamma e^{-\Im\lambda|x-y|}\,|x-y|^{-\frac 12}\,d\sigma(y)
\le C\,\la\Im\lambda\ra^{-\frac 12}\,.
$$
\end{proof}


\section{Resolvent Bounds in the Upper Half Plane}\label{sec:upperhalfplane}

In this section, we prove Theorems \ref{thm:optimal} and \ref{thm:optimalLargeIm} for $\Im \lambda >0$. 

We assume that $\Gamma$ is a finite union of compact subsets of embedded $C^{1,1}$ hypersurfaces, with induced surface measure. For $f\in L^2(\Gamma)$ we use $f\delta_\Gamma=\gamma^*f$ to denote the induced compactly supported distribution.

For $\Im\lambda>0$ let $R_0(\lambda)=(-\Delta- \lambda^2)^{-1}$ be the operator with Fourier multiplier $(|\xi|^2-\lambda^2)^{-1}$. 
For the proof of both Theorem \ref{thm:optimal} and Theorem \ref{thm:optimalLargeIm} we will estimate 
\begin{equation}
\label{eqn:restrictedDualEstimate}
Q_{\lambda}(f,g):=\int R_0(\lambda)(f\delta_\Gamma)\,\overline{g}\delta_\Gamma\,.
\end{equation}
For $\Im\lambda>0$, the right hand side \eqref{eqn:restrictedDualEstimate} agrees with the distributional pairing of $R_0(\lambda)(f\delta_\Gamma)\in H^{\frac 32-\e}$ with $g\delta_\Gamma\in H^{-\frac 12-\e}\,,$ and hence by the Plancherel theorem
\begin{equation}
\label{eqn:restrictedDualPlancherel}
Q_{\lambda}(f,g)=\int\frac{\widehat{f\delta_\Gamma}(\xi)\,\overline{\widehat{g\delta_\Gamma}}(\xi)}{|\xi|^2-\lambda^2}\,d\xi\,.
\end{equation}
For $|\lambda|\le 2$, the uniform bounds
$$
|Q_{\lambda}(f,g)|\le C\,\|f\|_{L^2(\Gamma)}\,\|g\|_{L^2(\Gamma)}\,,\qquad|\lambda|\le 2\,,
$$
follow easily from $f\delta_\Gamma\,, g\delta_\Gamma\in H^{-\frac 12-\epsilon}(\re^d)$, so we
focus on $|\lambda|\ge 2$.

We start by showing that resolvent bounds for $\lambda$ in the upper half plane can be deduced from restriction bounds for 
$\widehat{f\delta_\Gamma}$. Indeed, the following equivalence holds with $\delta_\Gamma$ replaced by any regular measure supported on a compact set.

\begin{lemma}
Suppose that for some $\alpha\in(0,1)$ the following estimate holds for $r>0$,
\begin{equation}
\label{eqn:fourierRestrictionEstimate}
\int\left|\widehat{f\delta_\Gamma}(\xi)\right|^2\delta(|\xi|-r)\leq C\,\la r\ra^\alpha \|f\|_{L^2(\Gamma)}^2\,.
\end{equation}
Then, for $\lambda$ in the upper half plane with $|\lambda|\geq 2$, 
$$
|Q_\lambda(f,g)|\leq  C\,|\lambda|^{\alpha-1}\log |\lambda|\,\|f\|_{L^2(\Gamma)}\|g\|_{L^2(\Gamma)}\,,
$$
where $Q_\lambda$ is as in \eqref{eqn:restrictedDualEstimate}.
\end{lemma}
\begin{proof}
Consider first the integral in \eqref{eqn:restrictedDualPlancherel} over $\bigl||\xi|-|\lambda|\bigr|\ge 1$. Since $\bigl||\xi|^2-\lambda^2\bigr|\ge \bigl||\xi|^2-|\lambda|^2\bigr|$, by the Schwartz inequality and \eqref{eqn:fourierRestrictionEstimate}
this piece of the integral is bounded by
\begin{align*}
\|f\|_{L^2(\Gamma)}\|g\|_{L^2(\Gamma)}
\int_{|r-|\lambda||\ge 1}
\la r\ra^{\alpha}\,\bigl|\,r^2-|\lambda|^2\,\bigr|^{-1}dr
\leq C\,|\lambda|^{\alpha-1}\log |\lambda|\,\|f\|_{L^2(\Gamma)}\|g\|_{L^2(\Gamma)}.
\end{align*}
Next, if $\Im\lambda\ge 1$, then $\bigl||\xi|^2-\lambda^2\bigr|\ge |\lambda|$, and by 
\eqref{eqn:fourierRestrictionEstimate}
$$
\Biggl|\;\int_{||\xi|-|\lambda||\le 1}
\frac{\widehat{f\delta_\Gamma}(\xi)\,\overline{\widehat{g\delta_\Gamma}}(\xi)}{|\xi|^2-\lambda^2}\,d\xi\;\Biggr|\le C\,|\lambda|^{\alpha-1}\,\|f\|_{L^2(\Gamma)}\|g\|_{L^2(\Gamma)}.
$$
Thus, we may restrict our attention to $0\le\Im \lambda\leq 1$ and $\bigl||\xi|-|\lambda|\bigr|\le 1$. 
For this piece we use that \eqref{eqn:fourierRestrictionEstimate} implies 
\begin{equation}
\label{eqn:fourierRestrictionEstimate2}
\int\left|\nabla_{\xi}\,\widehat{f\delta_\Gamma}(\xi)\right|^2\delta(|\xi|-r)\leq C\,\la r\ra^\alpha\|f\|_{L^2(\Gamma)}^2\,,
\end{equation}
due to the compact support of $f\delta_\Gamma$. 

We consider $\Re\lambda\ge 0$, the other case following similarly, and write
$$
\frac{1}{|\xi|^2-\lambda^2}=\frac{1}{|\xi|+\lambda}\;\frac{\xi}{|\xi|}\cdot\nabla_\xi\log(|\xi|-\lambda)\,,
$$
where the logarithm is well defined since $\Im(|\xi|-\lambda)<0$. Let $\chi(r)=1$ for $|r|\le 1$ and vanish for $|r|\ge \frac 32$. We then use integration by parts, together with
\eqref{eqn:fourierRestrictionEstimate} and \eqref{eqn:fourierRestrictionEstimate2} to bound
\begin{equation*}
\Biggl|\;\int\chi(|\xi|-|\lambda|)\,\frac{1}{|\xi|+\lambda}\,
\widehat{f\delta_\Gamma}(\xi)\,\overline{\widehat{g\delta_\Gamma}}(\xi)\,\;\frac{\xi}{|\xi|}\cdot\nabla_\xi\log(|\xi|-\lambda)\,d\xi\;\Biggr|
\le C\,|\lambda|^{\alpha-1}\,\|f\|_{L^2(\Gamma)}\|g\|_{L^2(\Gamma)}.
\end{equation*}
\end{proof}

To conclude the proof of Theorem \ref{thm:optimal}, we need to show that \eqref{eqn:fourierRestrictionEstimate} holds with $\alpha=\frac 12$ when $\Gamma$ is a compact subset of a $C^{1,1}$ hypersurface, and with $\alpha=\frac 13$ for a compact subset of a strictly convex $C^{2,1}$ hypersurface.
Since we work locally we assume that $\Gamma$ is given by the graph $x_n=F(x')$, where by an extension argument we assume that $F$ is a $C^{1,1}$ function (respectively $C^{2,1}$ function) defined on $\re^n$, and we replace surface measure on $\Gamma$ by $dx'$.
By scaling we may assume that $|\nabla F|\le\frac 1{20}$. We may also assume that $F(0)=0$.

Let $r\ge 1$, and let $\delta_{\sph^{d-1}_r}=\delta(|\xi|-r)$ be surface measure on the sphere $\sph^{d-1}_r$ of radius $r$. Assume that $g(\xi)$ is a function belonging to $L^2(\sph^{d-1}_r)$, and define
$$
Tg(x)=\int e^{i\langle x,\xi\rangle} g(\xi)\,\delta(|\xi|-r)\,.
$$
Let $\chi(x')\in C_c^\infty(\re^{d-1})$ be supported in the unit ball.
By duality, \eqref{eqn:fourierRestrictionEstimate} with $\alpha=\frac 12$ is equivalent to the following estimate
\begin{equation}\label{restrict}
\biggl(\int\bigl|(Tg)(x',F(x'))\bigr|^2\,\chi(x')\,dx'\biggr)^{\hf}\le C\,r^{\frac 14}\,\|g\|_{L^2(\sph^{d-1}_r)}\,,
\end{equation}
and for $\alpha=\frac 13$ is equivalent to the same estimate with $r^{\frac 14}$ replaced by $r^{\frac 16}$.

The estimate \eqref{restrict} is known as a restriction estimate for eigenfunctions of the Laplacian. $L^p$ generalizations in the setting of a smooth Riemannian manifold, with restriction to a smooth submanifold, were studied by Burq, G\'erard and Tzvetkov in \cite{BGT}. The $L^2$ estimates, again in the smooth setting, were noted by Tataru \cite{Tat} as being a corollary of an estimate of Greenleaf and Seeger \cite{GS}. These estimates were generalized to the setting of restriction to smooth submanifolds in Riemannian manifolds with metrics of $C^{1,1}$ regularity by Blair \cite{Blair}. In making a change of coordinates to flatten a submanifold the resulting metric has one lower order of regularity, thus the estimates of \cite{Blair} do not apply directly to $C^{1,1}$ submanifolds, and so we include here the proof of the $L^2$ estimate on $C^{1,1}$ hypersurfaces of Euclidean space. The estimate for strictly convex $C^{2,1}$ hypersurfaces does follow from \cite{Blair}, so we consider here just the case of a general $C^{1,1}$ hypersurface and $\alpha=\frac 12$.

We derive \eqref{restrict} from the following square function estimate for solutions to the wave equation.
\begin{lemma}\label{sqfnlemma}
Suppose that $f\in L^2(\re^d)$ and $\hat f(\xi)$ is supported in the region 
$\frac 34 r \le |\xi|\le \frac 32 r$. 
Then
\begin{equation}\label{sqfn'}
\biggl(\,\int_0^1\;\Bigl\|\Bigl(\cos(t\sqrt{-\Delta})f\Bigr)(x',F(x'))\Bigr\|_{L^2(\re^{d-1},dx')}^4\,dt\biggr)^{\frac 14}
\le C\,r^{\frac 14}\|f\|_{L^2(\re^d)}\,.
\end{equation}
\end{lemma}
The reduction of \eqref{restrict} to Lemma \ref{sqfnlemma} is attained by letting 
$f=\psi Tg$, where $\psi\in C_c^\infty(\re^d)$ equals 1 on the ball of radius 3. 
Then $\cos(t\sqrt{-\Delta})f= \cos(tr)Tg$ for $|x|<2$ and $|t|<1$.
On the other hand,
$\hat f=\hat\psi*\bigl(g\,\delta_{\sph^{d-1}_r}\bigr)$ is rapidly decreasing away from the sphere $|\xi|=r$, so the
difference between $\hat f$ and its truncation to $\frac 34r\le |\xi|\le \frac 32r$ is easily handled. Also, a simple calculation shows that, uniformly over $r$,
$$
\|\hat\psi*(g\,\delta_{\sph^{d-1}_r})\|_{L^2(\re^d)}\le C\,\|g\|_{L^2(\sph^{d-1}_r)}\,.
$$

\begin{proof}[Proof of Lemma \ref{sqfnlemma}.]
Given a function $F_r$ such that $\sup_{x'}|F_r(x')-F(x')|\le r^{-1}$, then \eqref{sqfn'} holds if we can show that
\begin{equation}\label{sqfn}
\biggl(\,\int_0^1\;\Bigl\|\Bigl(\cos(t\sqrt{-\Delta})f\Bigr)(x',F_r(x'))\Bigr\|_{L^2(\re^{d-1},dx')}^4\,dt\biggr)^{\frac 14}
\le C\,r^{\frac 14}\|f\|_{L^2(\re^d)}\,.
\end{equation}
This follows from the fact that \eqref{sqfn}, together with the frequency localization of $f$, implies the gradient bound, uniformly over $s$,
$$
\biggl(\,\int_0^1\;\Bigl\|\partial_s\Bigl(\cos(t\sqrt{-\Delta})f\Bigr)(x',F_r(x')+s)\Bigr\|_{L^2(\re^{d-1},dx')}^4\,dt\biggr)^{\frac 14}
\le C\,r^{\frac 54}\|f\|_{L^2(\re^d)}\,.
$$

We will take $F_r$ to be a mollification of the $C^{1,1}$ function $F$ on the $r^{-\hf}$ spatial scale. 
Precisely, let $F_r=\phi_{r^{1/2}}*F$, where 
$\phi_{r^{1/2}}=r^{\frac {d-1}2}\phi(r^\hf x)$, with $\phi$ a Schwartz function of integral 1.
Then
$$
\sup_x\,|F_r(x)-F(x)|\le C\,r^{-1}\,,\qquad\sup_x\,|\nabla F_r(x)-\nabla F(x)|\le C\,r^{-\hf}\,,
$$
and $F_r$ is a smooth function with derivative bounds
\begin{equation}\label{Fbounds}
\sup_x\,|\partial_x^\alpha F_r(x)|\le C\,r^{\frac{|\alpha|-2}2}\,,\qquad |\alpha|\ge 2\,.
\end{equation}

In establishing \eqref{sqfn} we may replace $\cos(t\sqrt{-\Delta})$ by $\exp(it\sqrt{-\Delta})$.  We then use a $TT^*$ argument to reduce to proving mapping properties
for an operator on $\Gamma_r\times [0,1]$. 
Precisely, let $K_r(t-s,x-y)$ denote the
kernel of the operator
$$
\rho\bigl(r^{-1}D)\exp\bigl(i(t-s)\sqrt{-\Delta}\,\bigr)\,,\qquad D:=-i\partial\,,
$$
where $\rho$ is a smooth function supported in the region $\frac 12<|\xi|<2$.
It then suffices to show that
\begin{multline}
\label{eqn:intermediate}
\biggl\|\int_0^1\int K_r\bigl(t-s,(x'-y',F_r(x')-F_r(y'))\bigr)\,f(s,y')\,dy'ds\biggr\|_{L^4([0,1],L^2(\re^{d-1}))}\\
\le C\,r^{\hf}\|f\|_{L^{4/3}([0,1],L^2(\re^{d-1}))}\,.\rule{0pt}{15pt}
\end{multline}
By the Hardy-Littlewood-Sobolev inequality
$$
\|t^{-\sigma}*f\|_{L^q(\re)}\leq C\|f\|_{L^p(\re)}\quad \text{where}\quad\frac 1q+1=\frac{\sigma}{d}+\frac{1}{p}\,.
$$ 
Hence translation invariance in $t$ shows that \eqref{eqn:intermediate} is a consequence of the following fixed-time
estimate, for $|t|<1$,
\begin{equation}\label{fixedtime}
\biggl\|\int K_r\bigl(t,(x'-y',F_r(x')-F_r(y'))\bigr)\,f(y')\,dy'\biggr\|_{L^2(\re^{d-1})}\\
\le C\,r^{\hf}\,|t|^{-\hf}\|f\|_{L^2(\re^{d-1})}\,.\rule{0pt}{15pt}
\end{equation}

If $|t|\le r^{-1}$, then $K_r$ satisfies
$$
|K_r(t,x-y)|\le C_N\,r^d\bigl(1+r\,|x-y|\bigr)^{-N}\,,
$$
and \eqref{fixedtime} follows by the Schur test. 
To prove \eqref{fixedtime} for $|t|>r^{-1}$, we decompose the convolution kernel $K_r(t,\cdot)$ as an almost orthogonal sum of terms, each of which behaves as a normalized convolution operator.
Fix $t\in [r^{-1},1]$, and let $\delta=r^{\hf} t^{-\hf}\,.$ 
Let $\eta_j$ count the elements of the lattice of spacing $\delta$ for which $|\eta_j|\in[\frac 12 r,2r]$,
and write
$$
\rho(r^{-1}\xi)=\sum_j Q_j(\xi)\,,
$$
where $Q_j$ is supported in the cube of sidelength $\delta$ centered on $\eta_j$,
and the following bounds hold on the derivatives of $Q_j$, uniformly over $r$, $t$ and $j$,
\begin{equation}\label{derivbounds}
\bigl|\partial_\xi^\alpha Q_j(\xi)\bigr|\le C_{\alpha} \, \delta^{-|\alpha|}\,.
\end{equation}
We then write $K_r(t,x)=\sum K_j(x)\,,$ where we suppress the dependence on $r$ and $t$, and set
$$
K_j(x)=(2\pi)^{-d}\int e^{i\langle x,\xi\rangle+it|\xi|}\,Q_j(\xi)\,d\xi\,.
$$
The multiplier $t|\xi|-t|\eta_j|^{-1}\langle\eta_j,\xi\rangle$ satisfies the derivative bounds 
\eqref{derivbounds} on the support of $Q_j$, hence we may write
$$
e^{i\langle x,\xi\rangle+it|\xi|}\,Q_j(\xi)=
e^{i\langle x+t\,|\eta_j|^{-1}\eta_j,\xi\rangle}\,\tilde Q_j(\xi)\,,
$$
with $\tilde Q_j$ having the same support and derivative conditions as $Q_j$.
Consequently, we may write
$$
K_j(x)=\delta^d\,e^{i\langle x,\eta_j\rangle+it|\eta_j|}
\chi_j\bigl(\delta(x+t\,|\eta_j|^{-1}\eta_j)\bigr)\,,
$$
where $\chi_j$ is a Schwartz function, with seminorm bounds independent of $j$.
We let 
$$
\Kbar_j(x',y')=K_j\bigl(x'-y',F_r(x')-F_r(y')\bigr)\,.
$$
It follows from the Schur test that
$$
\|\Kbar_j\|_{L^2\rightarrow L^2}\le C\,\delta\,.
$$
To handle the sum over $j$ we establish the estimate
\begin{equation}\label{cotlar}
\|\Kbar_j\Kbar_i^*\|_{L^2\rightarrow L^2}+\|\Kbar_j^*\Kbar_i\|_{L^2\rightarrow L^2}
\le C_N\,\delta^2\bigl(1+\delta^{-1}|\eta_i-\eta_j|\bigr)^{-N}\,,
\end{equation}
from which the bound \eqref{fixedtime} follows by the Cotlar-Stein lemma.
Since $K_j$ and $K_j^*$ have similar form, we restrict attention to the first term in \eqref{cotlar}.

The kernel $(K_jK_i^*)(x',z')$ has absolute value dominated by
$$
\delta^{2d}\int
\bigl(1+\delta\,|x+t\,|\eta_j|^{-1}\eta_j-y|\,\bigr)^{-N}
\bigl(1+\delta\,|z+t\,|\eta_i|^{-1}\eta_i-y|\,\bigr)^{-N}\,dy'
$$
where we use the notation $y=(y',F_r(y'))$, and similarly for $x$ and $z$.

Suppose that $|(\eta_j)_n|\ge \frac 14|\eta_j|$. Then since $|F_r(x')-F_r(y')|\le \frac 1{10}|x'-y'|$,
$$
\bigl|x'+t\,|\eta_j|^{-1}\eta'_j-y'\bigr|+10\bigl|F_r(x')+t\,|\eta_j|^{-1}(\eta_j)_n-F_r(y')\bigr|\ge 
5 \,t\,,
$$
and the Schur test leads to the bound
$$
\|\Kbar_j\Kbar_i^*\|_{L^2\rightarrow L^2}\le
C_N\,\delta^2\bigl(1+\delta\,t\bigr)^{-N}\,,
$$
which is stronger than \eqref{cotlar} since $|\eta_i-\eta_j|\le 6 r$. The same estimate
holds if $|(\eta_i)_n|\ge \frac 14|\eta_i|$.

We thus assume that $|(\eta_j)_n|\le \frac 14|\eta_j|$, and similarly for $\eta_i$.
Consider then the case where $|(\eta_i-\eta_j)_n|\ge |(\eta_i-\eta_j)'|$. Then we have
$$
\bigl|(|\eta_j|^{-1}\eta_j-|\eta_i|^{-1}\eta_i)_n\bigr|\ge
\frac 1{2+2\sqrt 2}\,\bigl|(|\eta_j|^{-1}\eta_j-|\eta_i|^{-1}\eta_i)'\bigr|\,,
$$
and since $\frac 12 r\le|\eta_i|,|\eta_j|\le 2r$,
$$
\bigl|(|\eta_j|^{-1}\eta_j-|\eta_i|^{-1}\eta_i)_n\bigr|\ge \frac{1}{4\sqrt 2}\,r^{-1}
|\eta_i-\eta_j|\,.
$$
Then with $|\nabla F_r|\le \frac 1{10}$,
\begin{multline*}
|x'-z'+t(|\eta_j|^{-1}\eta_j-|\eta_i|^{-1}\eta_i)'|+10\,|F_r(x')-F_r(z')+t(|\eta_j|^{-1}\eta_j-|\eta_i|^{-1}\eta_i)_n|\\
\ge \frac 5{4\sqrt 2}\,\delta^{-2}|\eta_j-\eta_i|\,,
\end{multline*}
hence using just the absolute bounds on the kernels, the operator norm is seen to be bounded by $C_N\,\delta^2\bigl(1+\delta^{-1}|\eta_j-\eta_i|\bigr)^{-N}$ as desired.

We thus consider the case that $|(\eta_j-\eta_i)_n|\le |(\eta_j-\eta_i)'|$. Up to a factor of modulus 1, the kernel $(K_jK_i^*)(x',z')$ can be written as
$$
\delta^{2d}\int
e^{-i\langle y',\eta_j'-\eta_i'\rangle-iF_r(y')(\eta_j-\eta_i)_n}\,
\chi_j\bigl(\delta(x+t\,|\eta_j|^{-1}\eta_j-y)\bigr)\,
\overline{\chi_i}\bigl(\delta(z+t\,|\eta_i|^{-1}\eta_i-y)\bigr)
\,dy'\,,
$$
where again $y=(y',F_r(y'))$, and similarly for $x$ and $z$. Since $|\nabla F_r(y')|\le \frac 1{10}$, and $|(\eta_j-\eta_i)_n|\le |\eta_j'-\eta_i'|\,,$ we have
$$
|\eta_j'-\eta_i'+\nabla F_r(y')(\eta_j-\eta_i)_n|\ge \tfrac 12 |\eta_j-\eta_i|\,.
$$
Using the estimates \eqref{Fbounds}, and that $r^\hf\le \delta$, an integration by parts argument dominates the kernel $(K_jK^*_i)(x',z')$ by
$$
\delta^{2d}\bigl(1+\delta^{-1}|\eta_j-\eta_i|\bigr)^{-N}\int
\bigl(1+\delta\,|x+t\,|\eta_j|^{-1}\eta_j-y|\,\bigr)^{-N}
\bigl(1+\delta\,|z+t\,|\eta_i|^{-1}\eta_i-y|\,\bigr)^{-N}\,dy'\,,
$$
which leads to the desired norm bounds, concluding the proof of \eqref{cotlar}.
\end{proof}

For the proof of Theorem \ref{thm:optimalLargeIm}, first consider the case that $f=g$ and
$\Gamma$ is a graph $x_n=F(x')$, and $\Im\lambda\ge 1$. We then have uniform bounds
$$
\sup_{\xi_n}\int\bigl|\widehat{f\delta_\Gamma}(\xi',\xi_n)\bigr|^2\,d\xi'\le C\,\|f\|^2_{L^2(\Gamma)}\,.
$$
We use the lower bound $\bigl||\xi|^2-\lambda^2\bigr|\ge |\lambda|\,|\Im\lambda|$ to dominate
$$
\int_{|\xi_n|\le 2|\lambda|} 
\frac{\bigl|\widehat{f\delta_\Gamma}(\xi)\bigr|^2}
{\bigl||\xi|^2-\lambda^2\bigr|}
\,d\xi\le C\,\la\Im\lambda\ra^{-1}\,\|f\|^2_{L^2(\Gamma)}\,.
$$
For $|\xi_n|\ge 2|\lambda|$ we have $\bigl||\xi|^2-\lambda^2\bigr|\gtrsim |\xi_n|^2\,,$
hence
$$
\int_{|\xi_n|\ge 2|\lambda|} 
\frac{\bigl|\widehat{f\delta_\Gamma}(\xi)\bigr|^2}
{\bigl||\xi|^2-\lambda^2\bigr|}
\,d\xi\le C\,\la\lambda\ra^{-1}\,\|f\|^2_{L^2(\Gamma)}\,.
$$
The case $f\ne g$ and $\Gamma$ a finite union of graphs follows by a partition of unity argument and the Schwarz inequality.


\section{Resolvent Bounds in the Lower Half Plane}\label{sec:lowerhalfplane}
For $\lambda\in \re$, the resolvent $R_0(\lambda)$ is defined as the limit $R_0(\lambda+i0)$ from $\Im\lambda>0$.
The estimates of the previous sections then give that, for $\lambda\in\re $ with $|\lambda|>2$, and for some $a>0$ and $b\in\{0,1\}$, we have 
$$
\|\gamma\, R_0(\lambda)\,\gamma^*\|_{L^2(\Gamma)\to L^2(\Gamma)}\leq 
C\,|\lambda|^{-a}\,\bigl(\log|\lambda|\bigr)^b\,.
$$ 
In this section we extend this to bounds for $\Im \lambda<0$. 
\begin{lemma}
Suppose that for $\lambda\in\re$, $|\lambda|>2$, the following holds
$$
\|\gamma \,R_0(\lambda)\,\gamma^*\|_{L^2(\Gamma)\to L^2(\Gamma)}
\leq C\,|\lambda|^{-a}(\log|\lambda|)^b\,.
$$
Then for $\Im\lambda\le 0\,,$ $|\lambda|\ge 2\,,$ and $\arg\lambda\in[-\pi,0]\cup[\pi,2\pi]$ in the case that $d$ is even,
$$
\|\gamma \,R_0(\lambda)\,\gamma^*\|_{L^2(\Gamma)\to L^2(\Gamma)}
\leq C\,|\lambda|^{-a}(\log|\lambda|)^be^{-\LGamma\Im \lambda}
$$
where $\LGamma$ is the diameter of $\Gamma$.
\end{lemma}
\begin{proof}
First consider the case that $d$ is odd. Suppose that 
$\|f\|_{L^2(\Gamma)}=\|g\|_{L^2(\Gamma)}=1$, and consider the function 
$$
F(\lambda)=e^{-i\LGamma\lambda}\,\lambda^a\,\bigl(\log\lambda)^{-b}\,Q_\lambda(f,g)\,,\qquad
\Im\lambda\le 0\,,\;\;|\lambda|\ge 2\,,
$$
where $\log\lambda$ is defined for $\arg\lambda\in(\frac \pi 2,\frac {5\pi}2)\,.$
Then $|F(\lambda)|\le C$ for $\lambda\in \re\setminus[-2,2]$ and for $|\lambda|=2\,.$ 
On the other hand, the resolvent kernel bounds \eqref{R0bounds} and the Schur test show that
$|F(\lambda)|$ has at most polynomial growth in $\lambda$ for $\Im\lambda \le 0$, since the
kernel $|x-x'|^{2-d}$ is integrable over a $d-1$ dimensional hypersurface. It follows by
the Phragm\'en-Lindel\"of theorem that $|F(\lambda)|\le C$ in the lower half plane.

In the case that $d$ is even, we note that the bounds of the lemma hold for $R_0(\lambda)$ if $\arg\lambda=2\pi$ and $|\lambda|\ge 2$. This follows since $R_0(e^{i\pi}\lambda)-R_0(\lambda)$
satisfies the same bounds as $R_0(\lambda)$ for $\arg\lambda=0$, and by \eqref{R0diff} we have $R_0(e^{2i\pi}\lambda)-R_0(e^{i\pi}\lambda)=R_0(e^{i\pi}\lambda)-R_0(\lambda)\,.$ We may thus apply the Phragm\'en-Lindel\"of theorem on the sheet $\pi\le\arg\lambda\le 2\pi\,.$ A similar
argument works for $-\pi\le\arg\lambda\le 0\,.$
\end{proof}


\section{Application to Resonance Free Regions}
\label{sec:equivalence}

In this section we establish Theorem \ref{thm:resFreeEasy}.
First, we demonstrate the meromorphic continuation of $R_V(\lambda)$ from $\Im\lambda\gg 0$ to $\lambda\in\complex$ (to the logarithmic cover in even dimensions) following arguments similar to those in the case where $V\in L^\infty_{\comp}$. We assume $\Gamma$ is a finite union of compact subsets of $C^{1,1}$ hypersurfaces, and that $\rho\in C_c^\infty(\re^d)$ with $\rho= 1$ on a neighborhood of $\Gamma$. Let
$$
K(\lambda)=(V\otimes\delta_\Gamma) R_0(\lambda)\,.
$$
The operator $K(\lambda)\rho:H^{-1}(\re^d)\rightarrow H^{-\frac 12-\e}_{\comp}$ is compact on $H^{-1}(\re^d)$ by Rellich's embedding theorem.
Furthermore, $I+K(\lambda)\rho$ is invertible if $\Im\lambda\gg 0$. To see this, note that $g+K(\lambda)\rho g=0$ and $g\in H^{-1}(\re^d)$ implies that $g=f\delta_\Gamma$ where
$f\in L^2(\Gamma)$. It follows that $f+VG(\lambda)f=0$, which implies $f=0$ for $\Im\lambda\gg 0$ by Theorem \ref{thm:optimalLargeIm}.

Consequently $(I+K(\lambda)\rho)^{-1}$ continues to a meromorphic family of Fredholm operators on $H^{-1}(\re^d)$ for $\lambda\in \complex$ (or to the logarithmic cover in even dimensions); see e.g.~Proposition 7.4 of \cite[Chapter 9]{Taylor}. Since $K=\gamma^*V\gamma R_0$, we have that
$$
(I+K(\lambda)\rho)^{-1}\gamma^*=\gamma^*(I+VG(\lambda))^{-1}\,,
$$
where $(I+VG(\lambda))^{-1}$ acts on $L^2(\Gamma)$. Hence, 
\begin{align*}
(I+K(\lambda)\rho)^{-1}&=I-(I+K(\lambda)\rho)^{-1}K(\lambda)\rho\\
&=I-\gamma^*(I+VG(\lambda))^{-1} V\gamma R_0(\lambda)\rho\,.
\end{align*}
The meromorphic extension of the resolvent $R_V(\lambda)$ for $-\Delta_{V,\Gamma}$ then equals
\begin{align*}
R_V(\lambda)&=R_0(\lambda)\,\bigl(I+K(\lambda)\rho\bigr)^{-1}\,\bigl(I-K(\lambda)(1-\rho)\bigr)\\
&=\Bigl( R_0(\lambda)-R_0(\lambda)\gamma^*(I+VG(\lambda))^{-1}V\gamma R_0(\lambda)\rho \Bigr)\,\bigl(I-K(\lambda)(1-\rho)\bigr)\,.
\end{align*}
In particular, if $g\in H^{-1}_{\comp}$ we can take $\rho g=g$ to obtain
\begin{equation}\label{eqn:resolventform}
R_V(\lambda)g=R_0(\lambda)g-R_0(\lambda)\gamma^*(I+VG(\lambda))^{-1}V\gamma R_0(\lambda)g\,.
\end{equation}
Consequently, $R_V(\lambda): H^{-1}_{\comp}\rightarrow H^1_{\loc}$, and its image is $\lambda$-outgoing.

The resolvent set $\Lambda$ is defined as the set of poles of $R_V(\lambda)$. Since
$$
\bigl(I-K(\lambda)(1-\rho)\bigr)\bigl(I+K(\lambda)(1-\rho)\bigr)=I\,,
$$
the preceding arguments show that $\Lambda$ agrees with the poles of $(I+VG(\lambda))^{-1}$, which by the Fredholm property agrees with the set of $\Lambda$ for which $(I+VG(\lambda))$ has nontrivial kernel.
If $\|G(\lambda)\|_{L^2\rightarrow L^2}< C^{-1}$ where $C=\|V\|_{L^2\rightarrow L^2}$, then $I+VG(\lambda)$ is invertible by Neumann series. By Theorem \ref{thm:optimal}, when $\Im\lambda<0$ this is the case provided that (for a different $C$)
$$
|\Im\lambda|\le \LGamma^{-1}\bigl(a\ln|\lambda|-\ln C-b\ln(\ln|\lambda|)\bigr)\,,
$$
which completes the proof of Theorem \ref{thm:resFreeEasy}.

We now observe that if $f$ solves $(I+VG(\lambda))f=0$ with $f\in L^2(\Gamma)$, then
\begin{equation}\label{uintegralrep}
u=R_0(\lambda)\bigl(f\delta_\Gamma\bigr)
\end{equation}
is a $\lambda$-outgoing solution to $-\Delta_{V,\Gamma}u=\lambda^2 u$.
Indeed, $u$ is $\lambda$-outgoing by definition, $u\in H^1_{\loc}$, and $(-\Delta-\lambda^2)u=f\delta_\Gamma\,.$ On the other hand,
$(V\otimes\delta_\Gamma)u=\bigl(VG(\lambda)f\bigr)\delta_\Gamma=-f\delta_{\Gamma}$.

To see that all such solutions arise this way we use the following extension of the Rellich uniqueness theorem, that there are no global $\lambda$-outgoing solutions to $(-\Delta-\lambda^2)u=0$. 
To prove this, note that for $0<\arg\lambda< \pi$ and $g$ a compactly supported distribution , $R_0(\lambda)g$ is exponentially decreasing in $|x|$, so  Green's identities yield, for $u=R_0(\lambda) g$ and for $R\gg 0$, that
$$
u(x)=
\int_{|y|=R}\Bigl(
G_0(\lambda,x,y)\,\partial_{\nu'}u(y)-
\partial_{\nu'_y}G_0(\lambda,x,y)u(y)
\Bigr)\,d\sigma(y)\,,\qquad |x|>R\,.
$$
By analytic continuation this holds for all $\lambda$. If $u$ is an entire solution then the right hand side is real-analytic in $R$, and we may let $R\rightarrow 0$ to deduce that $u\equiv 0$.

Suppose that $u\in H^1_{\loc}$ is a  $\lambda$-outgoing solution to $-\Delta_{V,\Gamma}u=\lambda^2 u$. By the uniqueness theorem it follows that 
\begin{equation}\label{eqn:unique}
u=-R_0(\lambda)(V\otimes\delta_\Gamma)u
=-\int_\Gamma G_0(\lambda,x,y)\,(Vu)(y)\,.
\end{equation}
Hence if
$f=Vu|_\Gamma$, then $f+VG(\lambda)f=0\,.$ By \eqref{eqn:unique} the correspondence between $u$ and $Vu|_\Gamma$ is one-to-one. It follows that the corresponding space of solutions $u$ for any $\lambda$ is finite dimensional, since it is in one-to-one correspondence with the kernel of a Fredholm operator.

Suppose now that $\Gamma=\partial\Omega$ for a compact domain $\Omega\subset\re^d$ with $C^{1,1}$ boundary. Assume also that $V\,:\,H^{\frac 12}(\partial\Omega)\rightarrow H^{\frac 12}(\partial\Omega)$. Then the analysis leading to \eqref{eqn:domainChar} shows that $u=u_1\oplus u_2$ satisfies the transmission problem \eqref{eqn:main}. Conversely, suppose $u=u_1\oplus u_2$ belongs to $\mcal E_2$ and satisfies \eqref{eqn:main}. For $w\in C_c^\infty(\re^d)$, Green's identities yield
$$
\int_{\re^d} u\,(-\Delta-\lambda^2)w=\int_{\partial\Omega}(\partial_\nu u+\partial_{\nu'}u)\,w=
-\int_{\partial\Omega}(Vu)\,w\,.
$$
Hence $u$ is a $\lambda$-outgoing $H^1_{\loc}$ distributional solution to 
$(-\Delta-\lambda^2)u+(V\otimes\delta_{\partial\Omega})u=0$, and by the above $\lambda$ is a resonance.


\section{Resonance Expansion for the Wave Equation}\label{sec:resExpand}

In this section we prove Theorems \ref{thm:resExpand} and \ref{thm:higherorder}. Let $\Lambda$ denote the set of resonances; since we work in odd dimensions $\Lambda$ is a discrete subset of $\complex$.
The elements of $\Lambda$ such that $\Im\lambda>0$ consist of $i\mu_j$ where $-\mu_j^2$ are the non-zero eigenvalues of $-\Delta_{V,\Gamma}$. That there are only a finite number of eigenvalues follows by relative compactness of $V\otimes\delta_\Gamma$ with respect to $-\Delta$. The resolvent near $i\mu_j$ takes the form
$$
R_V(\lambda)\;=\;\frac{-\Pi_{\mu_j}}{\lambda^2+\mu_j^2}+\text{holomorphic}\;=\;
\frac{i\,\Pi_{\mu_j}}{2\mu_j(\lambda-i\mu_j)}+\text{holomorphic}
\,,
$$
where $\Pi_{\mu_j}$ is projection onto the $-\mu_j^2$-eigenspace of $-\Delta_{V,\Gamma}$.
In particular we note that
\begin{equation}\label{eqn:mujresid}
\text{Res}\bigl(e^{-it\lambda}R_V(\lambda),i\mu_j\bigr)=i(2\mu_j)^{-1}e^{t\mu_j}\,\Pi_{\mu_j}\,.
\end{equation}

In dimension $d=1$, if $0\in\Lambda$ it is not an eigenvalue, whereas for $d\ge 5$ the corresponding solutions to \eqref{eqn:mainDelta} for $\lambda=0$ must be square-integrable. For $d=3$, if $0\in \Lambda$ there may be square-integrable and/or non square-integrable solutions to \eqref{eqn:mainDelta}, depending on whether the corresponding $f=Vu|_\Gamma$ in \eqref{uintegralrep} has vanishing integral.

For $|\lambda| \ll 1$ and $\Im\lambda>0$, the spectral bound $\|R_V(\lambda)\|_{L^2\rightarrow L^2}\le C (|\lambda|\Im\lambda\bigr)^{-1}$ shows that
$$
R_V(\lambda)=-\frac{\Pi_0}{\lambda^2}+\frac {i\mcal P_0}{\lambda}+\text{holomorphic}\,,
$$
where by inspection $\Pi_0$ is projection onto the $0$-eigenspace of $-\Delta_{V,\Gamma}$.
Since $R_V^*(-\overline\lambda)=R_V(\lambda)$ for $\Im\lambda>0$, it follows that $\mcal P_0$ is a symmetric map of $L^2_{\comp}$ to solutions of \eqref{eqn:mainDelta} with $\lambda=0$.
Hence,
\begin{equation}\label{eqn:0resid}
\text{Res}\bigl(e^{-it\lambda}R_V(\lambda),0\bigr)=
it\,\Pi_0+i\mcal P_0\,.
\end{equation}

In contrast to the case of $V\in L^\infty_{\comp}$, there may be resonances $\lambda\in \re\setminus\{0\}$. For an example in one dimension of $V$ and $\Gamma$ with such resonances consider $\Gamma=\left\{-\tfrac{\pi}{2},0,\tfrac{\pi}{2}\right\}$, and $V$ given by
$$(Vu|_{\Gamma})(x)=\begin{cases}u(0)\,,&x=\pm \frac{\pi}{2}\\
u(\tfrac \pi 2)+u(-\tfrac \pi 2)\,,&x=0
\end{cases}
$$
Then the function
$$u(x)=\begin{cases}\cos(x)\,,&|x|\leq \tfrac{\pi}{2}\\
0\,,&|x|\ge \tfrac{\pi}{2}
\end{cases}
$$
is a resonant state, that is an outgoing solution to \eqref{eqn:mainDelta}, for $\lambda=\pm 1$.

In fact, all resonances in $\re\setminus\{0\}$ must correspond to compactly supported eigenfunctions of $-\Delta_{V,\Gamma}$. To see this, suppose that $\lambda\in\re\setminus\{0\}$ and let $u\in\mcal D_{\loc}$ solve $-\Delta_{V,\Gamma}u=\lambda^2 u$. For $R\gg 0$, writing
$$
0=\int_{|x|\le R}\overline u\,(-\Delta u+(V\otimes\delta_{\Gamma})u-\lambda^2u)=
\int_{|x|\le R} \bigl(\,|\nabla u|^2-\lambda^2 |u|^2\bigr)+\int_{|x|=R}\overline u\,\partial_\nu u+\int_\Gamma \overline u \,Vu
$$
shows that $\Im\int_{|x|=R}\overline u\,\partial_\nu u=0$. The proof of Proposition 1.1 and Lemma 1.2 of \cite[Chapter 9]{Taylor} then show that $u\equiv 0$ on $|x|\ge R_0$, hence by analytic continuation $u$ vanishes on the unbounded component of $\re^d\setminus\Gamma$.

We note that if $\Gamma$ coincides with the boundary of the unbounded component of $\re^d\setminus\Gamma$ then there are no resonances $\lambda\in \re\setminus\{0\}$, since $0=u|_\Gamma$ implies that $(V\otimes\delta_{\Gamma})u=0$. Hence $u$ is a compactly supported eigenfunction of $-\Delta$ on $\re^d$, and must vanish identically.

The resonances in $\re\setminus\{0\}$ form a finite set by Theorem \ref{thm:resFreeEasy}, where $\lambda\in\re\setminus \{0\}$ is a resonance if $\lambda^2$ is an eigenvalue. The real resonances are thus symmetric about 0. We indicate them by $\pm\nu_j$, with $\nu_j>0$. By inspection, for $\Im\lambda>0$ near $\pm\nu_j$ we have
$$
R_V(\lambda)\;=\;\frac{-\Pi_{\nu_j}}{\lambda^2+\nu_j^2}+\text{holomorphic}\;=\;
\frac{\mp \Pi_{\nu_j}}{2\nu_j(\lambda\mp\nu_j)}+\text{holomorphic}
\,,
$$
where $\Pi_{\nu_j}$ is projection onto the $\nu_j^2$ eigenspace, hence
\begin{equation}\label{eqn:realresid}
\text{Res}\bigl(e^{-it\lambda}R_V(\lambda),\pm\nu_j\bigr)=
\mp (2\nu_j)^{-1}e^{\mp it\nu_j}\,\Pi_{\nu_j}\,.
\end{equation}

The remaining resonances form a discrete set $\{\lambda_j\}\subset\{\Im\lambda<0\}$, 
with respective multiplicity $m_R(\lambda_j)$.
Since $\lambda_j\ne 0$, the Laurent expansion of $R_V(\lambda)$ about 
$\lambda_j$ can be written in the following form
$$
R_V(\lambda)\;=\;i\!\!\sum_{k=1}^{m_R(\lambda_j)}
\frac{(-\Delta_{V,\Omega}-\lambda_j^2)^{k-1}\mcal P_{\lambda_j}}{(\lambda^2-\lambda_j^2)^k}
+\text{holomorphic}\,.
$$
Here
$\mcal P_{\lambda_j}\,:\,L^2_{\comp}\rightarrow\Dloc$ is given by
$$
\mcal P_{\lambda_j}=-\frac{1}{2\pi}\oint_{\lambda_j} R_V(\lambda)\,2\lambda\,d\lambda\,,
$$
and $(-\Delta_{V,\Omega}-\lambda_j^2)^{m_R(\lambda_j)}\mcal P_{\lambda_j}=0\,.$ We can thus write
\begin{equation}\label{eqn:jresid}
\text{Res}\bigl(e^{-it\lambda}R_V(\lambda),\lambda_j\bigr)\;=\;
i\!\!\!\!\!
\sum_{k=0}^{m_R(\lambda_j)-1}t^k\,e^{-it\lambda_j}\,\mcal P_{\lambda_j,k}
\end{equation}
where $\mcal P_{\lambda_j,k}\,:\,L^2_{\comp}\rightarrow\Dloc$. 
When $k=m_R(\lambda_j)-1$, $\mcal{P}_{\lambda_j,k}g$ is $\lambda_j$-outgoing, as seen by writing the Laurent expansion of $R_V(\lambda)$ in terms of that for $(I+K(\lambda)\rho)^{-1}$.
In particular, if $m_R(\lambda_j)=1$, then $\text{Res}\bigl(e^{-it\lambda}R_V(\lambda),\lambda_j\bigr)=i(2\lambda_j)^{-1}e^{-it\lambda_j}\,\mcal P_{\lambda_j}$, where $\mcal P_{\lambda_j}$ maps $L^2_{\comp}$ to $\lambda_j$-outgoing solutions of $(-\Delta_{V,\Omega}-\lambda_j^2)u=0\,.$ 


\subsection{Resolvent Estimates}
We first establish bounds on the cutoff of $R_V(\lambda)$, for $\lambda$ in the resonance free region established in Section \ref{sec:equivalence}.

\begin{lemma}
\label{lem:resolventEstimate}
Suppose that $\Gamma$ is a finite union of compact subsets of $C^{1,1}$ hypersurfaces.
Then for all $\e>0$ there exists $R<\infty$, so that if $\chi\in C_c^\infty(\re^d)$ equals $1$ on a neighborhood of $\Gamma$,  $\Re\lambda>R$, and
$\Im \lambda\ge - (\tfrac{1}{2}\LGamma^{-1}-\e)\log(\Re \lambda)$, 
then
\begin{align*}
\|\chi R_V(\lambda)\chi g\|_{L^2}&\;\le\;
C\,\la\lambda\ra^{-1}\,e^{2\Lchi(\Im \lambda)_-}\|g\|_{L^2}\,,\\
\|\chi R_V(\lambda)\chi g\|_{H^1}\!&\;\le\; 
C\,e^{2\Lchi(\Im \lambda)_-}\|g\|_{L^2}\,,\\
\|\chi R_V(\lambda)\chi g\|_{\mcal D}\,&\;\le\; 
C\,\la\lambda\ra\,e^{2\Lchi(\Im \lambda)_-}\|g\|_{L^2}\,,\rule{0pt}{13pt}
\end{align*}
where 
$\Lchi=\mathrm{diam}(\supp\chi)$, and $R_V(\lambda)$ is the meromorphic continuation of
$(-\Delta_{V,\Gamma}-\lambda^2)^{-1}$ from $\Im\lambda\gg 0$. If $\Im\lambda\ge 1$, $|\!\Re\lambda|>R$, then the estimates
hold with $\chi\equiv 1$, setting $\Lchi(\Im\lambda)_-=0$.
\end{lemma}

\noindent{\bf Remark:} The region in which this estimate is valid can be improved by replacing $\frac 12$ by $\frac 23$ if the components of $\Gamma$ are subsets of strictly convex $C^{2,1}$ hypersurfaces.

\medskip

\begin{proof}
We recall the Sobolev estimates for the cutoff of the free resolvent, see e.g. \cite[Chapter 3]{ZwScat}
\begin{equation*}
\|\chi R_0(\lambda)\chi\|_{H^s\rightarrow H^t}\le C\la\lambda\ra^{t-s-1}e^{\Lchi(\Im\lambda)_-}\,,\qquad s\le t\le s+2\,.
\end{equation*}
In addition, when $\Im\lambda\ge 1$ these estimates hold globally, that is with $\chi\equiv 1$.

This in turn leads to the following restriction estimates
\begin{equation}\label{eqn:traceest2}
\begin{split}
\|\gamma R_0(\lambda)\chi g\|_{L^2(\Gamma)}&\le
C\,\la\lambda\ra^{-s-\frac 12}e^{\Lchi(\Im\lambda)_-}\|g\|_{H^s}\,, 
\quad -\tfrac 32<s<\tfrac 12\,,\\
\|\gamma \nabla R_0(\lambda)\chi g\|_{L^2(\Gamma)}&\le
C\,\la\lambda\ra^{-s+\frac 12}e^{\Lchi(\Im\lambda)_-}\|g\|_{H^s}\,, 
\quad -\tfrac 12<s<\tfrac 32\,.
\rule{0pt}{17pt}
\end{split}
\end{equation}
To prove \eqref{eqn:traceest2} we use the following interpolation bound separately on each component of $\Gamma$,
$$
\|\gamma g\|_{L^2(\Gamma)}\le C_{t,t'}\,\|g\|_{H^t}^{\theta}\;\|g\|_{H^{t'}}^{1-\theta}\,,\qquad 0\le t<\tfrac 12<t'\,,\quad \theta(t-\tfrac 12)+(1-\theta)(t'-\tfrac 12)=0\,.
$$
By duality
we have the following extension estimate, 
\begin{equation}\label{eqn:extest2}
\|\chi R_0(\lambda)\gamma^* f\|_{H^s}\le 
C\,\la\lambda\ra^{s-\frac 12}\,e^{\Lchi(\Im\lambda)_-}\|f\|_{L^2(\Gamma)}\,,
\quad -\tfrac 12 <s < \tfrac 32\,.
\end{equation}
By restriction, note that \eqref{eqn:extest2} implies
\begin{equation}\label{eqn:GHhalfbound}
\|G(\lambda)f\|_{H^{1/2}(\Gamma)}\le
C_\e\,\la\lambda\ra^{\frac 12}\,e^{(\LGamma+\e)(\Im\lambda)_-}\|f\|_{L^2(\Gamma)}\,, \quad\e>0\,,
\end{equation}
where the norm on the left is the sum of the $H^{\frac 12}$ norms on the distinct $C^{1,1}$ components of $\Gamma$.

Now fix $g\in L^2(\re^d)$, set
$u=R_V(\lambda)\chi g\,.$  Then by \eqref{eqn:resolventform} we have $u=R_0(\lambda)\chi g-w$, where
$$
w=R_0(\lambda)\gamma^*(I+VG(\lambda))^{-1}V\gamma R_0(\lambda)\chi g\,.
$$
By Theorem \ref{thm:optimal}, for $|\!\Re \lambda|$ large enough and 
$\Im \lambda \geq -(\tfrac{1}{2}\LOmega^{-1}-\e)\log (|\!\Re\lambda|)$, the operator $I+VG(\lambda)$ is invertible on $L^2(\Gamma)$, and 
we have
$$
\|(I+VG(\lambda))^{-1}\|_{L^2(\Gamma)\to L^2(\Gamma)}\leq C\,,\qquad
\|VG(\lambda)\|_{L^2(\Gamma)\to L^2(\Gamma)}< 1\,.
$$

Thus, for $-\frac 32< s<\frac 12\,,$
\begin{equation*}
\| (I+VG(\lambda))^{-1}V\gamma R_0(\lambda)\chi g  \|_{L^2(\Gamma)}
\le C\, \la \lambda\ra ^{-s-\frac 12}\,e^{\Lchi(\Im\lambda)_-}\|g\|_{H^s}\,.
\end{equation*}
Then \eqref{eqn:extest2} gives the following, for $-\frac 32< s<\frac 12\,,$ and with global bounds if $\Im\lambda\ge 1$,
\begin{align}
\|\chi w\|_{L^2}&\leq C\,\la \lambda\ra^{-s-1}e^{2\Lchi(\Im \lambda)_-}\|g\|_{H^s}\,,
\label{eqn:L2wbound}
\\
\|\chi w\|_{H^1}
&\leq C\,\la\lambda\ra^{-s}\,e^{2\Lchi(\Im \lambda)_-}\|g\|_{H^s}\rule{0pt}{13pt}\,.
\label{eqn:H1wbound}
\end{align}
By the $L^2\rightarrow H^t$ bounds for $\chi R_0(\lambda)\chi$ the same holds for $s=0$ with $w$ replaced by $u$, which yields the bounds of Lemma \ref{lem:resolventEstimate} except for the ones on $\|\chi u\|_{\mcal D}$.

To obtain bounds on $\|\chi u\|_{\mcal D}\,,$ we write
$$
\Delta(\chi u) =
-\chi g+2(\nabla\chi)\cdot\nabla u+(\Delta\chi)u-\lambda^2\chi u+(V\otimes\delta_{\Gamma}) u\,,
$$
and note by \eqref{eqn:L2wbound} and \eqref{eqn:H1wbound} that
$$
\|(\nabla\chi)\cdot\nabla u\|_{L^2}+\|(\Delta \chi)u\|_{L^2}+\la\lambda\ra^2\|\chi u\|_{L^2}\le C\,\la\lambda\ra\, e^{2\Lchi(\Im \lambda)_-}\|g\|_{L^2}\,.
$$
Consequently,
$$
\|\Delta_{V,\Gamma}(\chi u)\|_{L^2}\le C\,\la\lambda\ra\, e^{2\Lchi(\Im \lambda)_-}\|g\|_{L^2}\,,
$$
yielding the desired bound on $\|\chi u\|_{\mcal D}$.
\end{proof}


\subsection{Proof of Theorem \ref{thm:resExpand}}

We prove here the case $N=1$ of Theorem \ref{thm:resExpand}; the case $N\ge 2$ will be handled following the proof of Theorem \ref{thm:higherorder}.
We follow the treatment in \cite{TangZw}, and suppose that $g\in H^s$ for some $0<s<\frac 12$ and proceed by density in $L^2$.
As above write
\begin{align*}
R_V(\lambda) \chi g=w+R_0(\lambda) \chi g\,.
\end{align*}

Choose $\alpha\ge 1$ so that $\mu_j<\alpha$ for all $j$, where $-\mu_j^2$ are the negative eigenvalues of $-\Delta_{V,\Gamma}$. By the spectral theorem we can write
\begin{align}
U(t)  \chi g&=\frac{1}{2\pi}\int_{-\infty+i\alpha}^{\infty+i\alpha }e^{-it\lambda} R_V(\lambda)\chi g\, d\lambda\nonumber\\
&=\frac{1}{2\pi}\int_{-\infty+i\alpha}^{\infty+i\alpha }e^{-it\lambda} \bigl( w+R_0(\lambda)\chi g\bigr)\, d\lambda\,.\rule{0pt}{20pt}\label{eqn:integralConv}
\end{align}
The integral is norm convergent in $L^2(\re^d)$, by \eqref{eqn:L2wbound} and the norm convergence of the free resolvent integral.
After localizing by $\chi$ on the left, for $t$ sufficiently large we seek to deform the contour $\re+i\alpha$ to 
$$
\Sigma_A=\bigl\{\lambda\in \complex: \Im \lambda = -A-c \log\bigl(2+ |\!\Re \lambda|\,\bigr)\bigr\}
$$
where we choose $c<\frac 12\LGamma^{-1}$, and assume $A$ is such that there are no resonances on $\Sigma_A$.  
We will show that the integral over $\Sigma_A$ is norm convergent for $g\in H^s$ if $s>0$, so to justify the contour change we need to show that for $t$ sufficiently large the integrals over 
$$
\gamma_{\pm R}(v)=\left\{\pm R+iv\,:\,-\bigl(A+c\log(2+R)\bigr)\le v\le \alpha\right\}\,,\quad\text{and}\quad\gamma_{R,\infty}=\left\{x+i\alpha\,:\,|x|\geq R\right\}
$$
tend to $0$ as $R\rightarrow\infty$. Note that for $R$ large enough, Theorem \ref{thm:resFreeEasy} shows that there are no resonances between 
$\re+i\alpha$ and $\Sigma_A$ with
$|\!\Re\lambda|\ge R$, and hence none on $\gamma_{\pm R}$.

We introduce the following notation,
$$
E_\gamma(t)f=\frac{1}{2\pi}\int_{\gamma}e^{-it\lambda} R_V(\lambda) f \,d\lambda\,.
$$
Then for $t>2\Lchi$, and $R$ large enough,
$$
\|\chi E_{\gamma_{\pm R}}(t)\chi g\|_{L^2}\leq C\,e^{\alpha t}\la R\ra^{-1}\bigl(\alpha+A+c\log(2+ R)\bigr)\|g\|_{L^2}\to 0\quad\text{as}\quad
R\rightarrow\infty\,.
$$ 
The norm convergence of \eqref{eqn:integralConv} shows that $\|\chi E_{\gamma_{R,\infty}}\chi g\|_{L^2}\to 0$ as $R\rightarrow \infty$. 
We then assume $c(t-2\Lchi)\ge 3$ and calculate
$$
\|\chi E_{\Sigma_A}(t)\chi g\|_{\mcal D}\leq C_{A,\chi}\,e^{-A(t-2\Lchi)}\int_{-\infty}^\infty 
e^{-3\log(2+|R|)}
\la A+|R|\,\ra\,dR \le C_{A,\chi}\,e^{-At}\,\|g\|_{L^2}\,.
$$
In particular the integral is norm convergent, and the contour deformation is allowed.

Thus, if we let $\Omega_A$ denote the collection of poles of $R_V(\lambda)$ in the set 
$\Im \lambda>-A-c \log\bigl(2+ |\!\Re \lambda|)$, then
$$
\chi U(t)\chi g=\chi E_{\Sigma_A}(t) \chi g-i\chi\sum_{z\in\Omega_A}
\text{Res}\bigl(e^{-it\lambda}R_V(\lambda),z)\chi g\,,
$$
and by density this holds for $g\in L^2(\re^d)$. Observe that if $g\in L^2_{\comp}$ then we can take $\chi=1$ on the support of $g$, and drop the cutoff $\chi$ to write a global equality in $L^2_{\loc}$.  To have estimates on the remainder in $\mcal D$, though, requires cutting off by $\chi$ and taking $t>2 \Lchi+C$, consistent with the propagation of singularities.
The expressions \eqref{eqn:mujresid}, \eqref{eqn:0resid}, \eqref{eqn:realresid}, and \eqref{eqn:jresid} now complete the proof of Theorem \ref{thm:resExpand} for $N=1$, where we observe that the terms from poles in $\Omega_A$ with $\Im\lambda\le -A$ can be absorbed into $E_A(t)$.
\qed


\subsection{Higher order estimates for smooth domains}
We start with the following lemma, where we now assume that $\Gamma=\partial\Omega$ is $C^\infty$, and that $V\,:\,H^s(\partial\Omega)\rightarrow H^s(\partial\Omega)$ for all $s\ge 0$. Recall that we set 
$\mcal E_0=L^2(\re^d)$, and for $N\ge 1$,
$$
\mcal E_N= H^1(\re^d)\cap \bigl( H^N(\Omega)\oplus H^N(\re^d\setminus\overline\Omega)\bigr)\,.
$$
In this setting $\mcal D$ equals the subspace of $\mcal E_2$ satisfying $\partial_\nu u+\partial_{\nu'}u+Vu|_{\partial\Omega}=0\,.$

\begin{lemma}\label{lem:higherorder}
Suppose that $\partial\Omega$ is of regularity $C^\infty$, and $N\ge 0$.
Then for all $\e>0$ there exists $R<\infty$, so that if $|\!\Re\lambda|>R$,
$|\Im \lambda|\le (\tfrac{1}{2}\LOmega^{-1}-\e)\log(|\!\Re \lambda|)$, and $\chi\in C_c^\infty(\re^d)$ equals $1$ on a neighborhood of $\overline\Omega$, then
$$
\bigl\|\chi \bigl(R_V(\lambda)-R_V(-\lambda)\bigr)\chi g\bigr\|_{\mcal E_N} \le 
C_N\,\la\lambda\ra^{N-1}\,e^{2\Lchi(\Im \lambda)_-}\|g\|_{L^2}\,.
$$
\end{lemma}
\begin{proof}
We proceed by induction on $N$. By Lemma \ref{lem:resolventEstimate}, the result holds for $N=0,1,2.$ We assume then that the result is true for integers less than or equal to $N$.

Letting $u= \bigl(R_V(\lambda)-R_V(-\lambda)\bigr)\chi f$, we write
$$
\Delta (\chi u) = 
2(\nabla\chi)\cdot\nabla u +(\Delta\chi)u- \lambda^2\chi u+(V\otimes \delta_{\partial\Omega}) u\,.
$$
By the induction hypothesis,
\begin{align*}
\|(\Delta\chi) u\|_{H^{N-1}(\Omega)\oplus H^{N-1}(\re^d\setminus\overline\Omega)}+\|\chi u\|_{H^{N-1}(\Omega)\oplus H^{N-1}(\re^d\setminus\overline\Omega)}&\le C\,\la\lambda\ra^{N-2}\,e^{2\Lchi(\Im \lambda)_-}\|g\|_{L^2}\,,\\
\|(\nabla\chi)\cdot\nabla u\|_{H^{N-1}(\Omega)\oplus H^{N-1}(\re^d\setminus\overline\Omega)}+\|Vu\|_{H^{N-\frac 12}(\partial\Omega)}&\le C\,\la\lambda\ra^{N-1}\,e^{2\Lchi(\Im \lambda)_-}\|g\|_{L^2}\,. \rule{0pt}{15pt}
\end{align*}
Lemma \ref{lem:smoothtransmission} then gives the desired result for $\mcal E_{N+1}$.
\end{proof}

\begin{proof}[Proof of Theorem \ref{thm:higherorder}]
We use the notation from the proof of Theorem \ref{thm:resExpand} above.
We first note that
$$
\frac {1}{2\pi}\int_{\Sigma_A} e^{-it\lambda} R_V(-\lambda)\,d\lambda\;=\;-\!\!\!
\sum_{\mu_j>A+\log 2}(2\mu_j)^{-1}e^{-t\mu_j}\Pi_{\mu_j}\,,
$$
where the completion of the contour to the lower half plane is justified by Lemma \ref{lem:resolventEstimate} and the rapid decrease of $e^{-it\lambda}$ for $t>0$.
We thus can write
$$
\chi E_{\Sigma_A}(t)\chi g=
\frac 1{2\pi}\int_{\Sigma_A} e^{-it\lambda} \chi \bigl(R_V(\lambda)-R_V(-\lambda)\bigr)\chi g\,d\lambda\;-\!\!\!
\sum_{\mu_j>A+\log 2}(2\mu_j)^{-1}e^{-t\mu_j}\chi \Pi_{\mu_j}\chi g\,.
$$
Assume $c(t-2\Lchi)\ge N+1\,,$ the $\mcal E_N$ norm of the integral term is dominated by
$$
C_{A,\chi}\,e^{-A(t-2\Lchi)}\int_{-\infty}^\infty 
e^{-(N+1)\log(2+|R|)}
\la A+|R|\,\ra^{N-1}\,dR \le C_{A,\chi,N}\,e^{-At}\,\|g\|_{L^2}\,.
$$

It remains to show that if $\mu_j>A$, and if $\Im \lambda_j<-A$, then
$$
e^{-t\mu_j}\|\chi \Pi_{\mu_j}\chi g\|_{\mcal E_N} + \|\chi \text{Res}\bigl(e^{-it\lambda}R_V(\lambda),\lambda_j\bigr)\chi g\|_{\mcal E_N}\le
C_{A,\chi,N}\,e^{-tA}\,\|g\|_{L^2}\,,
$$ 
since the difference of $\chi E_A(t)\chi$ and $\chi E_{\Sigma_A}(t)\chi$ is a sum of such terms. 

A similar argument to the proof of Lemma \ref{lem:higherorder} gives the bound
$$
\|\Pi_{\mu_j}f\|_{\mcal E_N}\le C_N\,\la\mu_j\ra^{N}\|f\|_{L^2}\,,
$$
which handles the eigenvalues.
To handle the resonances in the lower half plane, consider first the case that $-\lambda_j$ is not a pole. We can then write
$$
\text{Res}\bigl(e^{-it\lambda}R_V(\lambda),\lambda_j\bigr)=
\frac 1{2\pi i}\oint_{\lambda_j} e^{-it\lambda}\bigl( R_V(\lambda)-R_V(-\lambda)\bigr)\,d\lambda\,,
$$
and the estimate follows from Lemma \ref{lem:higherorder}, by choosing a small contour about $\lambda_j$ which is contained in $\Im\lambda<-A$.
In the case that $-\lambda_j$ is a pole, hence an eigenvalue, then the term $R_V(-\lambda)$ contributes an eigenvalue projection, which is handled as above. 
\end{proof}

We now complete the proof of Theorem \ref{thm:resExpand} by considering the case $N\ge 2$.
Eigenfunctions clearly belong to $\mcal D_N$, and by an induction argument we have $\|\chi\Pi_{\mu_j}\chi g\|_{\mcal D_N}\le C_N\la\mu_j\ra^{2N}\|g\|_{L^2}$. The proof then follows from that of Theorem \ref{thm:higherorder}, using the following

\begin{lemma}
Suppose that $\Gamma$ is a finite union of $C^{1,1}$ hypersurfaces, and $N\ge 1$.
Then for all $\e>0$ there exists $R<\infty$ so that if  $|\!\Re\lambda|>R$,
$|\Im \lambda|\le (\tfrac{1}{2}\LGamma^{-1}-\e)\log(|\!\Re \lambda|)$, and $\chi\in C_c^\infty(\re^d)$ equals $1$ on a neighborhood of $\Gamma$, 
then
\begin{align*}
\bigl\|\chi \bigl(R_V(\lambda)-R_V(-\lambda)\bigr)\chi g\bigr\|_{\mcal D_N} &\le 
C\,\la\lambda\ra^{2N-1}\,e^{2\Lchi(\Im \lambda)_-}\|g\|_{L^2}\,.
\end{align*}
\end{lemma}
\begin{proof}
The result was proven above for $N=1$. We then proceed by induction, writing
\begin{align*}
\Delta_{V,\Gamma}\,\chi\bigl(R_V(\lambda)-R_V(-\lambda)\bigr)\chi g&=
\Bigl(\bigl[\Delta,\chi\bigr]-\lambda^2\chi\Bigr)\bigl(R_V(\lambda)-R_V(-\lambda)\bigr)\chi g
\\
&=\Bigl(2\nabla\chi\cdot\nabla+(\Delta\chi)-\lambda^2\chi\Bigr)\bigl(R_V(\lambda)-R_V(-\lambda)\bigr)\chi g\,.
\end{align*}
By induction, and since $\text{supp}(\Delta\chi)\subset\text{supp}(\chi)$,
\begin{equation}\label{eqn:induction}
\|\bigl((\Delta\chi)-\lambda^2\chi\bigr)\bigl(R_V(\lambda)-R_V(-\lambda)\bigr)\chi g\|_{\mcal D_{N-1}}
\le C\,\la\lambda\ra^{2N-1}\,e^{2\Lchi(\Im \lambda)_-}\|g\|_{L^2}\,.
\end{equation}
On the complement of $\Gamma$, the function $u=\bigl(R_V(\lambda)-R_V(-\lambda)\bigr)\chi g$
satisfies $-\Delta u=\lambda^2u\,,$ and by Lemma \ref{lem:resolventEstimate}, if $\chi_1\in C_c^\infty$ with $\text{supp}(\chi_1)\subset\text{supp}(\chi)$,
$$
\la\lambda\ra\,\|\chi_1 u\|_{L^2}+
\|\chi_1 u\|_{H^1}\le C\,e^{2\Lchi(\Im \lambda)_-}\|g\|_{L^2}\,.
$$
Since $\nabla\chi$ vanishes on a neighborhood of $\Gamma$, an induction argument and elliptic regularity yields
$$
\|\nabla \chi\cdot\nabla \bigl(R_V(\lambda)-R_V(-\lambda)\bigr)\chi g \|_{H^{2N-1}}\le C\,\la\lambda\ra^{2N-1}e^{2\Lchi(\Im \lambda)_-}\|g\|_{L^2}\,,\quad N\ge 1\,.
$$
Since $H^{2N-1}_{\comp}(\re^d\setminus\Gamma)\subset \mcal D_{N-1}$ with continuous inclusion, this term also satisfies the bound of \eqref{eqn:induction}, and the result follows.
\end{proof}


\section{Appendix: the Transmission Property for $C^{1,1}$ Domains}\label{sec:transmission}

We provide here a proof of the transmission estimate that we need to establish $H^2$ regularity of solutions away from $\partial\Omega$. In the case of smooth boundaries, the following is well known; see \cite{BdM}, and in particular Theorems 9 and 10 of \cite{Epstein}.

\begin{lemma}\label{lem:smoothtransmission}
Suppose that $\Omega\subset\re^d$ is a bounded open set, and $\partial\Omega$ is locally the graph of a $C^\infty$ function. Let $G_0(x,y)$ be Green's kernel for $\Delta^{-1}$. Then for $N\ge -1$ the single layer potential map
$$
S\ell f(x)=\int_{\partial\Omega} G_0(x,y)\,f(y)\,d\sigma(y)
$$
induces a continuous map from $H^{N+\frac 12}(\partial\Omega)$ to $H^{N+2}(\Omega)\oplus H^{N+2}_{\loc}(\re^d\setminus\overline\Omega)$. 

Additionally, for $N\ge 0$ the map
$$
G_0 g(x)=\int G_0(x,y)\,g(y)\,dy
$$
induces a continuous map from $H^N(\Omega)\oplus H^N_{\comp}(\re^d\setminus\overline\Omega)$ to $H^{N+2}(\Omega)\oplus H^{N+2}_{\loc}(\re^d\setminus\overline\Omega)$.
\end{lemma}
We need the same result for $N=0$ and $\partial\Omega$ of $C^{1,1}$ regularity, in which case just the single layer potential result is nontrivial.

\begin{lemma}\label{lem:transmission}
Suppose that $\Omega\subset\re^d$ is a bounded open set, and $\partial\Omega$ is locally the graph of a $C^{1,1}$ function. Let $G_0(x,y)$ be Green's kernel for $\Delta^{-1}$. Then the single layer potential map
$$
S\ell f(x)=\int_{\partial\Omega} G_0(x,y)\,f(y)\,d\sigma(y)
$$
induces a continuous map from $H^{\frac 12}(\partial\Omega)$ to $H^2(\Omega)\oplus H^2_{\loc}(\re^d\setminus\overline\Omega)$.
\end{lemma}
\begin{proof}
Since the kernel is smooth away from the diagonal we may work locally, and assume that $\partial\Omega$ is given as a graph $x_n=F(x')\,,$ with $F\in C^{1,1}(\re^{d-1})$.
Since surface measure $d\sigma(y)=m(y')\,dy'$ where $m$ is Lipschitz, we can absorb the $m$ into $f$. Assuming then that $f\in C^1_c(\re^{d-1})$, consider the maps
\begin{align*}
T'\!f(x)=(\nabla_{x'} S\ell f)(x',F(x')+x_d)&\, = \,
c_d \int \frac{(x'-y')\,f(y')\,dy'}{\bigl(|x'-y'|^2+|x_d+F(x')-F(y')|^2\bigr)^{\frac d2}}\\
T_d f(x)=(\partial_{x_d} S\ell f)(x',F(x')+x_d)& \, = \,
c_d \int \frac{(x_d+F(x')-F(y'))\,f(y')\,dy'}{\bigl(|x'-y'|^2+|x_d+F(x')-F(y')|^2\bigr)^{\frac d2}}
\rule{0pt}{20pt}
\end{align*}
We seek $H^{\frac 12}\rightarrow H^1_{\loc}(x_d\ne 0)$ bounds for both terms. 
We have $\partial_{x_d}T'=\nabla_{x'}T_d-(\nabla_{x'} F)\partial_{x_d} T_d$, and
since $\Delta S\ell f=0$, for $x_d\ne 0$ we can write
$$
(1+|\nabla_{x'}F|^2)\partial_{x_d}T_df=\nabla_{x'}T'f-(\nabla_{x'} F)\nabla_{x'}T_df\,.
$$
Thus it suffices to prove $H^{\frac 12}\rightarrow L^2$ bounds for $\chi\nabla_x' T'$ and $\chi\nabla_{x'}T_d\,.$

By the dual of the trace estimate we have
\begin{equation*}
\|\chi S\ell f\|_{H^1}\le
C\,\|f\|_{H^{-1/2}(\partial\Omega)}\,,
\end{equation*}
and hence we can bound
$$
\|\chi \,T'\! (\nabla_{y'}f)\|_{L^2}+\|\chi\, T_d (\nabla_{y'}f)\|_{L^2}\le C\,\|f\|_{H^{1/2}(\partial\Omega)}\,.
$$
The desired bound will thus follow from showing that
$$
\bigl\|\chi\bigl[\nabla_{x'},T'\bigr]f\bigr\|_{L^2}+
\bigl\|\chi\bigl[\nabla_{x'},T_d\bigr]f\bigr\|_{L^2}\le C\,\|f\|_{L^2(\partial\Omega)}\,.
$$
Both maps can be written in the form
$\int K(x',x_d,y')\,f(y')\,dy'\,,$ where
$$
\sup_{x'}\int_{|y'|\le L}|K(x',x_d,y')|\,dy'+\sup_{y'}\int_{|x'|\le L}|K(x',x_d,y')|\,dx'\le
C_L\,\la\ln |x_d|\ra\,,
$$
from which the result follows by the Schur test.
\end{proof}



\begin{thebibliography}{10}

\bibitem{AS} M. Abramowitz and I. Stegun,
\newblock
{\em Handbook of mathematical functions with formulas, graphs, and mathematical tables.}
\newblock
Dover Publications, New York, 1972.

\bibitem{Aligia}
A. A. Aligia and A. M. Lobos,
\newblock 
Mirages and many-body effects in quantum corrals.
\newblock 
J. Phys.: Condens. Matter {\bf 17} 2005.


\bibitem{Blair} M. Blair,
\newblock
$L^q$ bounds on restrictions of spectral clusters to submanifolds for low regularity metrics.
\newblock 
To appear, Analysis and PDE.


\bibitem{Heller}
M. Barr, M. Zalatel, and E. Heller,
\newblock
Quantum corral resonance widths: lossy scattering as acoustics.
\newblock 
Nano Letters (2010), no. 10, p. 3253--3260.

\bibitem{BdM}
L. Boutet de Monvel,
\newblock
Comportement d'un op\'erateur pseudo-diff\'erentiel sur une vari\'et\'e \`a bord. I. La propri\'et\'e de transmission. (French)
\newblock
J. Analyse Math. {\bf 17} (1966), 241--253. 


\bibitem{BGT}
N. Burq, P. G\'{e}rard, and N. Tzvetkov,
\newblock
Restrictions of the Laplace-Beltrami eigenfunctions to submanifolds. 
\newblock
Duke Math. J. 138 (2007), no. 3, 445--486.

\bibitem{Burke}
P. G. Burke,
\newblock
\emph{Potential scattering in atomic physics}.
Plenum Press, New York, 1977.

\bibitem{Calderon}
A. Calder\'on,
\newblock
Lebesgue spaces of differentiable functions.
\newblock
Proc. Symp. in Pure Math. {\bf 4} (1961), 33--49.

\bibitem{Card}
F. Cardoso, G. Popov, and G. Vodev,
\newblock 
Distribution of resonances and local energy decay in the transmission problem. II.
\newblock 
Math. Res. Lett. 6 (1999), no. 3--4, 377--396.


\bibitem{Crommie}
M. Crommie, C. Lutz, D. Eigler, and E. Heller,
\newblock 
Quantum corrals.
\newblock
Physica D: Nonlinear Phenomena 83 (1995), no. 1, 98--108.

\bibitem{Epstein}
C. Epstein,
\newblock
Pseudodifferential methods for boundary value problems. 
\newblock
{\em Pseudo-differential operators: partial differential equations and time-frequency analysis},
\newblock
171--200, Fields Inst. Commun., 52, Amer. Math. Soc., Providence, RI, 2007.

\bibitem{Galk}
J. Galkowski,
\newblock
Distribution of resonances in lossy scattering. 
\newblock
In preparation.


\bibitem{GS}
A. Greenleaf and A. Seeger,
\newblock
Fourier integral operators with fold singularities.
\newblock
J. Reine Angew. Math. {\bf 455} (1994), 35--56.


\bibitem{Lax}
P. D. Lax and R. S. Phillips,
\newblock
\emph{Scattering theory. Second edition.}
\newblock
Pure and Applied Mathematics, 26. Academic Press, Inc., Boston, MA, 1989.


\bibitem{PopVod}
G. Popov and G. Vodev,
\newblock
Distribution of the resonances and local energy decay in the transmission problem.
\newblock 
Asymptot. Anal. 19 (1999), no. 3--4, 253--265.


\bibitem{RS}
M. Reed and B. Simon,
\newblock
\emph{Methods of modern mathematical physics I : functional analysis.}
\newblock
Academic Press, New York, 1980.


\bibitem{Stein} E. Stein,
\newblock
\emph{Singular integrals and differentiability properties of functions.}
\newblock
Princeton University Press, 1971.


\bibitem{TangZw}
S. Tang and M. Zworski,
\newblock
Resonance expansions of scattered waves.
\newblock 
Comm. Pure Appl. Math. 53 (2000), no. 10, 1305--1334. 
 

\bibitem{Tat}
D. Tataru,
\newblock
On the regularity of boundary traces for the wave equation.
\newblock 
Ann. Scuola Norm. Sup. Pisa Cl. Sci. (4) 26 (1998), no. 1, 185--206.


\bibitem{Taylor}
M. Taylor,
\newblock 
\emph{Partial differential equations II. Qualitative studies of linear equations. Second edition.}.
\newblock 
Applied Mathematical Sciences, 116. Springer, New York, 2011.

\bibitem{Vain}
B. Vainberg,
\newblock 
\emph{Asymptotic methods in equations of mathematical physics},
\newblock 
Gordon \& Breach, 1989.


\bibitem{ZwAMS}
M. Zworski,
\newblock 
Resonances in physics and geometry.
\newblock 
Notices Amer. Math. Soc. 46 (1999), no. 3, 319--328.

\bibitem{ZwScat}
M. Zworski,
\newblock 
\emph{Lectures on scattering resonances.}
\newblock 
http://math.berkeley.edu/$\sim$zworski/res.pdf.


\end{thebibliography}
\end{document}